\documentclass{article}

\usepackage{amsmath,amsthm,amssymb,float,bbm,titlesec,tikz,subcaption,float}
\usepackage[margin = 1in]{geometry}

\newtheorem{theorem}{Theorem}
\newtheorem{corollary}{Corollary}
\newtheorem{lemma}{Lemma}
\newtheorem{proposition}{Proposition}
\theoremstyle{definition}
\newtheorem*{ack}{Acknowledgments}

\newtheorem{example}{Example}
\newtheorem{note}{Note}

\titleformat{\section}[runin]
  {\normalfont\large\bfseries}{\thesection.}{1em}{}
\titleformat{\subsection}[runin]
  {\normalfont\normalsize\bfseries}{\thesubsection.}{1em}{}

\title{\Large Non-pseudounitary fusion}
\author{\normalsize Andrew Schopieray\thanks{This material is based upon work supported by the National Science Foundation under Grant No. DMS-1440140, while the author was in residence at the Mathematical Sciences Research Institute in Berkeley, California, during the Spring 2020 semester.}}
\date{\today}

\begin{document}
\maketitle

\begin{abstract}
We prove there exist infinitely many inequivalent fusion categories whose Grothendieck rings do not admit any pseudounitary categorifications.
\end{abstract}

\section{Introduction}
\label{sec:intro}

Fusion categories are both a generalization of the categories of representations of finite groups, and an algebraic axiomatization of the notion of quantum symmetry.  Technically, fusion categories $\mathcal{C}$ (over $\mathbb{C}$) are $\mathbb{C}$-linear semisimple tensor categories with finitely many isomorphism classes of simple objects, the set of which will be denoted $\mathcal{O}(\mathcal{C})$, and a tensor unit $\mathbbm{1}\in\mathcal{O}(\mathcal{C})$.  Evidence that this is an important area of research is that the same objects arise from a study of representation theory of Lie algebras, subfactors and planar algebras, vertex operator algebras, and conformal field theory.  In this sense, fusion categories are an inevitable result of late 20th and early 21st century mathematics and mathematical physics.

\par The skeleton of a fusion category $\mathcal{C}$ is its Grothendieck ring, $K(\mathcal{C})$, which is an example of a fusion ring \cite[Definition 3.1.7]{tcat}.  There are infinitely many fusion rings $R$ which do not arise in this fashion even with rank 2 \cite{MR1981895}, so we say a fusion category $\mathcal{C}$ is a \emph{categorification} of $R$ if $R=K(\mathcal{C})$, and reflexively that $R$ is \emph{categorifiable} if there exists a fusion category $\mathcal{C}$ such that $R=K(\mathcal{C})$.   Ring homomorphisms $\varphi:K(\mathcal{C})\to\mathbb{C}$ are referred to as dimension functions and provide a method for measuring the ``size'' of a fusion category via $\varphi(\mathcal{C}):=\sum_{X\in\mathcal{O}(\mathcal{C})}|\varphi(X)|^2$.  Every fusion category possesses the dimension function $\mathrm{FPdim}:K(\mathcal{C})\to\mathbb{C}$ \cite[Proposition 3.3.6(i)]{tcat} which for each $X\in\mathcal{O}(\mathcal{C})$ is computed as the Frobenius-Perron (maximal real) eigenvalue of the matrix of tensoring with $X$.  Fusion categories which possess a spherical structure \cite[Section 4.7]{tcat} have another dimension function $\dim:K(\mathcal{C})\to\mathbb{C}$ often referred to as categorical dimension.  When $\dim(\mathcal{C})=\mathrm{FPdim}(\mathcal{C})$ we say that $\mathcal{C}$ is \emph{pseudounitary}, and that $K(\mathcal{C})$ admits a \emph{pseudounitary categorification}.  Any weakly integral fusion category $\mathcal{C}$, i.e. \!$\mathrm{FPdim}(\mathcal{C})\in\mathbb{Z}$, is pseudounitary \cite[Proposition 8.24]{ENO} which includes representation categories of finite-dimensional quasi-Hopf algebras \cite[Theorem 8.33]{ENO}.  Pseudounitary fusion categories are privy to a vast range of tools and results which generic fusion categories are not.  For example \cite[Proposition 8.23]{tcat}, if $\mathcal{C}$ is a pseudounitary fusion category, there exists a unique spherical structure on $\mathcal{C}$ such that $\dim(X)=\mathrm{FPdim}(X)$ for all $X\in\mathcal{O}(\mathcal{C})$.  For these reasons and more, it is not surprising that a large proportion of the literature to date assumes, whether implicitly or explicitly, that the fusion categories in question are pseudounitary; the goal of this paper is to show that this assumption is far from innocuous.  We prove the set of fusion rings which admit a pseudounitary categorification is a strict subset of the set of all categorifiable fusion rings.

\begin{theorem}\label{the}
There exists a fusion category of rank $6$ whose Grothendieck ring does not admit any pseudounitary categorifications.
\end{theorem}

The proof of Theorem \ref{the} is outlined in Section \ref{sec:proof} while the specific details are relegated to Sections \ref{sec:fus}--\ref{sec:struc}; Corollary \ref{thecor} describes how infinitely many examples may be constructed from this initial one.   We assume the reader is familiar with the basic results of fusion and modular tensor categories found in a standard text such as \cite{tcat}.  Exotic examples are not needed to achieve this result.  Only examples from the representation theory of quantum groups at roots of unity are needed, a topic which is over 30 years old at this time.  The literature on these examples is so vast that we direct the reader to two expository papers on the subject \cite{rowell,MR4079742} where one can find condensed histories and extensive references.  The rank 6 fusion category we provide satisfying the conditions of Theorem \ref{the} we label $\mathcal{R}:=\mathcal{C}(\mathfrak{so}_5,3/2)_\mathrm{ad}$ for brevity, or $\mathcal{C}(\mathfrak{so}_5,9,q)_\mathrm{ad}$ in the notation of \cite{MR4079742} where $q=\exp(\pi i/9)$.   The category $\mathcal{R}$ is described in detail in Example \ref{one}; it is self-dual, i.e. \!$X\cong X^\ast$ for all $X\in\mathcal{O}(\mathcal{R})$, and has the structure of a \emph{modular tensor category}, i.e. \!$\mathcal{R}$ is a spherical fusion category equipped with a nondegenerate braiding.  Modular tensor categories of rank strictly less than 6 have been classified \cite{MR3632091,MR2544735} and all necessary underlying fusion rings have pseudounitary categorifications.  And so if the classification of modular tensor categories by rank is to continue, it must enter a new and inherently non-pseudounitary world.

\par There is a related notion of a \emph{unitary} (spherical/braided/modular) fusion category which requires morphism spaces be equipped with a positive-definite Hermitian form satisfying copious compatibility conditions with the existing categorical structures (e.g. \!\cite[Definition 2.18]{MR3367967}).  As a consequence, all unitary fusion categories are pseudounitary, but currently it is not known if these concepts are equivalent.  Nonetheless, Corollary \ref{the} trivially implies there exist infinitely many isomorphism classes of categorifiable fusion rings which admit no unitary categorifications as well.  Unitarity comes along naturally with many of the analytical frameworks which produce fusion categories such as subfactors and vertex operator algebras.  An analog of Corollary \ref{the} has already been proven in the unitary/braided setting \cite[Theorem 3.8(a)]{MR2414692}.  Specifically, it was shown that if $k\in\mathbb{Z}_{\geq2}$ and $q^2$ is a primitive $\ell^\text{\tiny{th}}$ root of unity such that $\ell$ is odd and $2k+5\leq\ell$, then neither $K(\mathcal{C}(\mathfrak{so}_{2k+1},q, \ell))$ nor $K(\mathcal{C}(\mathfrak{sp}_{\ell-2k-1},q,\ell))$ are categorifiable by a unitary braided fusion category.  Our Theorem \ref{the} implies that the assumptions of a braiding and unitarity in \cite[Theorem 3.8(a)]{MR2414692} can be replaced by the strictly weaker assumption of pseudounitarity in at least one case; it would be interesting to know if this true in general.  Lastly, it is of great importance to the classification of fusion categories whether fusion rings such as $K(\mathcal{R})$ are exceptional, or expected.  In particular, for every $M\in\mathbb{R}_{\geq1}$, there exist finitely many categorifiable fusion rings $R$ with $\mathrm{FPdim}(R)\leq M$ \cite[Corollary 3.13]{paul}.  What proportion of categorifiable fusion rings admit a pseudounitary categorification for large $M$?  If this proportion is small, then stronger non-pseudounitary tools will be crucial to the study fusion categories henceforth.

\par 

\begin{ack}
We would like to thank Victor Ostrik and Eric Rowell for posing the initial question which led to Theorem \ref{the}, and for reviewing early drafts.
\end{ack}


\section{Proof of Theorem \ref{the}}\label{sec:proof}

In Section \ref{sec:fus}, we describe the fusion category $\mathcal{R}:=\mathcal{C}(\mathfrak{so}_5,3/2)_\mathrm{ad}$ whose existence proves Theorem \ref{the}.   Proposition \ref{brayded} first states that there does not exist a pseudounitary \emph{braided} fusion category $\mathcal{C}$ with $K(\mathcal{C})=K(\mathcal{R})$.  To prove our claim without assuming a braiding exists, we pass to the Drinfeld center $\mathcal{Z}(\mathcal{C})$ \cite[Section 8.5]{tcat}, a modular tensor category.  Proposition \ref{last} then states that the dimensions of noninvertible elements of $\mathcal{O}(\mathcal{Z}(\mathcal{C}))$ generate $\mathbb{Q}(\zeta_9)^+$, the totally real cubic subfield of the cyclotomic field of ninth roots of unity.
\par In Section \ref{nums}, we classify all totally positive algebraic $d$-numbers of norm $3^n$ which are perfect squares in, and generate, $\mathbb{Q}(\zeta_9)^+$.  These are a superset of all possible squared dimensions of elements of $\mathcal{O}(\mathcal{Z}(\mathcal{C}))$.  In Section \ref{sec:pos}, Proposition \ref{dimensions} determines the rank of $\mathcal{Z}(\mathcal{C})$ and dimensions of $Z\in\mathcal{O}(\mathcal{Z}(\mathcal{C}))$ using the Galois action on the modular data, the induction functor $I:\mathcal{C}\to\mathcal{Z}(\mathcal{C})$ \cite[Section 9.2]{tcat}, and forgetful functor $F:\mathcal{Z}(\mathcal{C})\to\mathcal{C}$.  Lastly, Proposition \ref{prop:sub} proves that a simple object of $\mathcal{Z}(\mathcal{C})$ of smallest nontrivial dimension generates a braided fusion subcategory $\mathcal{D}\subset\mathcal{Z}(\mathcal{C})$ with $K(\mathcal{D})=K(\mathcal{R})$.  This is incompatible with Proposition \ref{brayded}, therefore no such pseudounitary categorification of $K(\mathcal{R})$ exists.

\par Elementary computations, particularly artithmetic in $\mathbb{Q}(\zeta_9)$, are performed using the open-source software SageMath \cite{sagemath}.  We have included the necessary code in Appendix \ref{sage} and in a supplementary text file.

\begin{corollary}\label{thecor}
There exist infinitely many fusion categories whose Grothendieck rings do not admit any pseudounitary categorifications.
\end{corollary}

\begin{proof} Let $\mathcal{R}$ be the rank 6 fusion category of Theorem \ref{the}, $\mathcal{R}'$ be any pseudounitary fusion category, and assume there exists a pseudounitary fusion category $\mathcal{C}$ with $K(\mathcal{C})=K(\mathcal{R}\boxtimes\mathcal{R}')=K(\mathcal{R})\times K(\mathcal{R}')$.  Then there exists a pseudounitary (as $\dim$ and $\mathrm{FPdim}$ are multiplicative) fusion subcategory $\mathcal{D}\subset\mathcal{C}$ with $K(\mathcal{C})=K(\mathcal{R})$, which cannot exist by Theorem \ref{the}.  Hence $K(\mathcal{R}\boxtimes\mathcal{R}')$ has no pseudounitary categorifications either.   The fact that infinitely-many inequivalent pseudounitary fusion categories $\mathcal{R}'$ exist, e.g. \!$\mathrm{Rep}(G)$ for any finite group $G$, completes our proof.
\end{proof}


\section{The category $\mathcal{R}:=\mathcal{C}(\mathfrak{so}_5,3/2)_\mathrm{ad}$}\label{sec:fus}
\begin{example}\label{one}
At least two styles of notation are used to represent the spherical braided fusion categories arising from quantum groups at roots of unity.  One is $\mathcal{C}(\mathfrak{g},\ell,q)$ where $q$ is a root of unity such that $q^2$ has order $\ell\in\mathbb{Z}_{\geq2}$, and another is $\mathcal{C}(\mathfrak{g},k)$ where $q^2=\exp(2\pi i/(m(k+h^\vee)))$ with $m=1,2,3$ the laceity of $\mathfrak{g}$, $h^\vee$ its dual Coxeter number, and $k\in\frac{1}{m}\mathbb{Z}$ referred to as the level of $\mathcal{C}(\mathfrak{g},k)$.  The former describes a larger set of categories than the latter.  When the order of $q^2$ is sufficiently large, $\mathcal{O}(\mathcal{C}(\mathfrak{g},k))$ is indexed by weights in a truncated rendition of the classical dominant Weyl chamber.  The geometry of this truncation is sensitive to the order of $q^2$ when $\mathfrak{g}$ is not simply-laced.  We illustrate this in Figure \ref{fig:g21e} for $\mathfrak{g}=\mathfrak{so}_5$ where the solid lines indicate the walls of the classical dominant Weyl chamber and the dotted line indicates the truncation.  When $q^2$ is a ninth root of unity, the truncation is made perpendicular to the short root, and when $q^2$ is a tenth root of unity, the truncation is made perpendicular to the long root \cite[Section 3.1]{rowell}.

\begin{figure}[H]
\centering
\begin{subfigure}{.5\textwidth}
  \centering
\begin{equation*}
\begin{tikzpicture}[scale=1]
\node at (0,-0.33) {\tiny $X_0$};
\node at (0,1-0.33) {\tiny $X_1$};
\node at (0,2-0.33) {\tiny $X_2$};
\node at (1,1-0.33) {\tiny $X_3$};
\node at (1,2-0.33) {\tiny $X_4$};
\node at (2,2-0.33) {\tiny $X_5$};

\node at (-1,-2) {$\cdot$};
\node at (0,-2) {$\cdot$};
\node at (1,-2) {$\cdot$};
\node at (2,-2) {$\cdot$};
\node at (3,-2) {$\cdot$};
\node at (4,-2) {$\cdot$};

\node at (-1/2,-3/2) {$\cdot$};
\node at (1/2,-3/2) {$\cdot$};
\node at (3/2,-3/2) {$\cdot$};
\node at (5/2,-3/2) {$\cdot$};
\node at (7/2,-3/2) {$\cdot$};
\node at (9/2,-3/2) {$\cdot$};

\node at (-1,-1) {$\cdot$};
\node at (0,-1) {$\cdot$};
\node at (1,-1) {$\cdot$};
\node at (2,-1) {$\cdot$};
\node at (3,-1) {$\cdot$};
\node at (4,-1) {$\cdot$};

\node at (-1/2,-1/2) {$\cdot$};
\node at (1/2,-1/2) {$\cdot$};
\node at (3/2,-1/2) {$\cdot$};
\node at (5/2,-1/2) {$\cdot$};
\node at (7/2,-1/2) {$\cdot$};
\node at (9/2,-1/2) {$\cdot$};

\node at (-1,0) {$\cdot$};
\node at (0,0) {$\blacklozenge$};
\node at (1,0) {$\cdot$};
\node at (2,0) {$\cdot$};
\node at (3,0) {$\cdot$};
\node at (4,0) {$\cdot$};

\node at (-1/2,1/2) {$\cdot$};
\node at (3/2,1/2) {$\cdot$};
\node at (5/2,1/2) {$\cdot$};
\node at (7/2,1/2) {$\cdot$};
\node at (9/2,1/2) {$\cdot$};

\node at (-1,1) {$\cdot$};
\node at (2,1) {$\cdot$};
\node at (3,1) {$\cdot$};
\node at (4,1) {$\cdot$};

\node at (-1/2,3/2) {$\cdot$};
\node at (5/2,3/2) {$\cdot$};
\node at (7/2,3/2) {$\cdot$};
\node at (9/2,3/2) {$\cdot$};

\node at (-1,2) {$\cdot$};
\node at (3,2) {$\cdot$};
\node at (4,2) {$\cdot$};

\node at (-1/2,5/2) {$\cdot$};
\node at (7/2,5/2) {$\cdot$};
\node at (9/2,5/2) {$\cdot$};

\node at (-1,3) {$\cdot$};
\node at (0,3) {$\cdot$};
\node at (1,3) {$\cdot$};
\node at (2,3) {$\cdot$};
\node at (3,3) {$\cdot$};
\node at (4,3) {$\cdot$};

\node at (-1/2,7/2) {$\cdot$};
\node at (1/2,7/2) {$\cdot$};
\node at (3/2,7/2) {$\cdot$};
\node at (5/2,7/2) {$\cdot$};
\node at (7/2,7/2) {$\cdot$};
\node at (9/2,7/2) {$\cdot$};

\node at (1/2,1/2) {$\lozenge$};
\node at (0,1) {$\blacklozenge$};
\node at (1,1) {$\blacklozenge$};
\node at (1/2,3/2) {$\lozenge$};
\node at (3/2,3/2) {$\lozenge$};
\node at (0,2) {$\blacklozenge$};
\node at (1,2) {$\blacklozenge$};
\node at (2,2) {$\blacklozenge$};
\node at (1/2,5/2) {$\lozenge$};
\node at (3/2,5/2) {$\lozenge$};
\node at (5/2,5/2) {$\lozenge$};
\draw[<->] (-1/2,-2) -- (-1/2,7/2);
\draw[dashed,<->] (-1,3) -- (9/2,3);
\draw[<->] (-1,-2) -- (9/2,7/2);
\end{tikzpicture}
\end{equation*}
  \caption{$k=3/2$}
  \label{fig:so532}
\end{subfigure}%
\begin{subfigure}{.5\textwidth}
  \centering
\begin{equation*}
\begin{tikzpicture}[scale=1]
\node at (-1,-2) {$\cdot$};
\node at (0,-2) {$\cdot$};
\node at (1,-2) {$\cdot$};
\node at (2,-2) {$\cdot$};

\node at (-1/2,-3/2) {$\cdot$};
\node at (1/2,-3/2) {$\cdot$};
\node at (3/2,-3/2) {$\cdot$};
\node at (5/2,-3/2) {$\cdot$};

\node at (-1,-1) {$\cdot$};
\node at (0,-1) {$\cdot$};
\node at (1,-1) {$\cdot$};
\node at (2,-1) {$\cdot$};

\node at (-1/2,-1/2) {$\cdot$};
\node at (1/2,-1/2) {$\cdot$};
\node at (3/2,-1/2) {$\cdot$};
\node at (5/2,-1/2) {$\cdot$};

\node at (-1,0) {$\cdot$};
\node at (0,0) {$\blacklozenge$};
\node at (1,0) {$\cdot$};
\node at (2,0) {$\cdot$};

\node at (-1/2,1/2) {$\cdot$};
\node at (3/2,1/2) {$\cdot$};
\node at (5/2,1/2) {$\cdot$};

\node at (-1,1) {$\cdot$};
\node at (2,1) {$\cdot$};

\node at (-1/2,3/2) {$\cdot$};
\node at (5/2,3/2) {$\cdot$};

\node at (-1,2) {$\cdot$};

\node at (-1/2,5/2) {$\cdot$};

\node at (-1,3) {$\cdot$};
\node at (0,3) {$\cdot$};
\node at (1,3) {$\cdot$};
\node at (2,3) {$\cdot$};

\node at (-1,4) {$\cdot$};
\node at (0,4) {$\cdot$};
\node at (1,4) {$\cdot$};
\node at (2,4) {$\cdot$};

\node at (-1/2,7/2) {$\cdot$};
\node at (1/2,7/2) {$\cdot$};
\node at (3/2,7/2) {$\cdot$};
\node at (5/2,7/2) {$\cdot$};

\node at (1/2,1/2) {$\lozenge$};
\node at (0,1) {$\blacklozenge$};
\node at (1,1) {$\blacklozenge$};
\node at (1/2,3/2) {$\lozenge$};
\node at (3/2,3/2) {$\cdot$};
\node at (0,2) {$\blacklozenge$};
\node at (1,2) {$\cdot$};
\node at (2,2) {$\cdot$};
\node at (1/2,5/2) {$\cdot$};
\node at (3/2,5/2) {$\cdot$};
\node at (5/2,5/2) {$\cdot$};
\draw[<->] (-1/2,-2) -- (-1/2,4);
\draw[dashed,<->] (-1,4) -- (5/2,1/2);
\draw[<->] (-1,-2) -- (5/2,3/2);
\end{tikzpicture}
\end{equation*}
  \caption{$k=2$}
  \label{fig:so52}
\end{subfigure}
\caption{$\mathcal{C}(\mathfrak{so}_5,k)_\mathrm{ad}$ ($\blacklozenge$) and $\mathcal{C}(\mathfrak{so}_5,k)$ ($\lozenge$ \& $\blacklozenge$)}
\label{fig:g21e}
\end{figure}
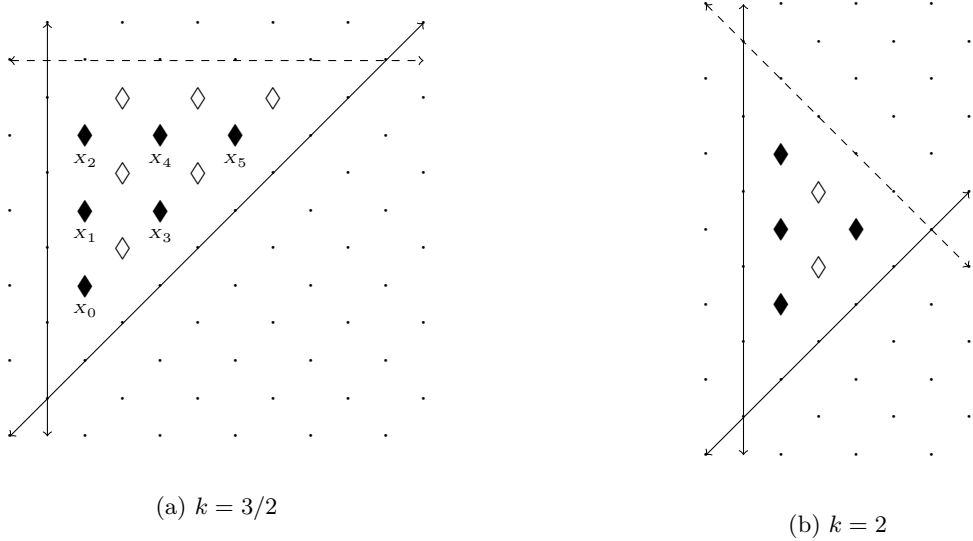  

The subject of this exposition is the adjoint subcategory \cite[Definition 4.14.5]{tcat} $\mathcal{R}:=\mathcal{C}(\mathfrak{so}_5,3/2)_\mathrm{ad}$ whose simple objects are indexed by dominant weights lying in the root lattice, and categorifications of its Grothendieck ring.  We will order the elements of $\mathcal{O}(\mathcal{R})$ as in \cite{rowell} by $X_0,X_1,X_2,X_3,X_4,X_5$ for consistency.  The fusion rules of $\mathcal{R}$ are tabulated here for reference where $(N_i)_{j,k}=\dim\mathrm{Hom}(X_i\otimes X_j,X_k)$.

\begin{align}
N_0&=\left[\begin{array}{cccccc}
1 & 0 & 0 & 0 & 0 & 0 \\
0 & 1 & 0 & 0 & 0 & 0 \\
0 & 0 & 1 & 0 & 0 & 0 \\
0 & 0 & 0 & 1 & 0 & 0 \\
0 & 0 & 0 & 0 & 1 & 0 \\
0 & 0 & 0 & 0 & 0 & 1
\end{array}\right]&
N_1&=\left[\begin{array}{cccccc}
0 & 1 & 0 & 0 & 0 & 0 \\
1 & 0 & 1 & 1 & 0 & 0 \\
0 & 1 & 0 & 0 & 1 & 0 \\
0 & 1 & 0 & 1 & 1 & 0 \\
0 & 0 & 1 & 1 & 1 & 1 \\
0 & 0 & 0 & 0 & 1 & 1
\end{array}\right]&
N_2&=\left[\begin{array}{cccccc}
0 & 0 & 1 & 0 & 0 & 0 \\
0 & 1 & 0 & 0 & 1 & 0 \\
1 & 0 & 0 & 1 & 0 & 1 \\
0 & 0 & 1 & 1 & 1 & 0 \\
0 & 1 & 0 & 1 & 1 & 1 \\
0 & 0 & 1 & 0 & 1 & 0
\end{array}\right]                \\
N_3&=\left[\begin{array}{cccccc}
0 & 0 & 0 & 1 & 0 & 0 \\
0 & 1 & 0 & 1 & 1 & 0 \\
0 & 0 & 1 & 1 & 1 & 0 \\
1 & 1 & 1 & 1 & 1 & 1 \\
0 & 1 & 1 & 1 & 2 & 1 \\
0 & 0 & 0 & 1 & 1 & 1
\end{array}\right]&
N_4&=\left[\begin{array}{cccccc}
0 & 0 & 0 & 0 & 1 & 0 \\
0 & 0 & 1 & 1 & 1 & 1 \\
0 & 1 & 0 & 1 & 1 & 1 \\
0 & 1 & 1 & 1 & 2 & 1 \\
1 & 1 & 1 & 2 & 2 & 1 \\
0 & 1 & 1 & 1 & 1 & 0
\end{array}\right]&
N_5&=\left[\begin{array}{cccccc}
0 & 0 & 0 & 0 & 0 & 1 \\
0 & 0 & 0 & 0 & 1 & 1 \\
0 & 0 & 1 & 0 & 1 & 0 \\
0 & 0 & 0 & 1 & 1 & 1 \\
0 & 1 & 1 & 1 & 1 & 0 \\
1 & 1 & 0 & 1 & 0 & 0
\end{array}\right]
\end{align}
The Frobenius-Perron eigenvalues of $N_0,N_1,N_2,N_3,N_4,N_5$ lie in the cyclotomic field $\mathbb{Q}(\zeta_9)^+$, the real cubic subfield of $\mathbb{Q}(\zeta_9)$ where $\zeta_n:=\exp(2\pi i/n)$ for $n\in\mathbb{Z}_{\geq1}$.  One can also describe $\mathbb{Q}(\zeta_9)^+$ as the splitting field of $x^3-3x-1$, whose maximal real root is $a:=2\cos(\pi/9)$.  The set $\{1,a,a^2\}$ is then an integral basis for $\mathcal{O}_{\mathbb{Q}(\zeta_9)^+}$, its ring of algebraic integers.  A multiplicative generating set for the unit group $\mathcal{O}^\times_{\mathbb{Q}(\zeta_9)^+}$, which exists by Dirichlet's unit theorem, is $\pm1$ along with the pair $u_1:=a^2-2=\zeta_9-\zeta_9^2-\zeta_9^5$ and $\tilde{u_2}:=u_1-a$.  This pair is not uniquely determined but is the pair computed, for instance, in \cite{MR866105}.  We will replace $\tilde{u_2}$ with $u_2:=-\tilde{u_2}^{-1}=1-\zeta_9^4-\zeta_9^5$ without loss of generality so that $u_1$ and $u_2$ are both $\geq1$.  With this notation, $\mathrm{FPdim}(X_0)=1$, $\mathrm{FPdim}(X_j)=u_2$ for $j=1,2,5$, $\mathrm{FPdim}(X_3)=u_1u_2$, and $\mathrm{FPdim}(X_4)=u_1^{-1}u_2^2$.  Their total sum of squares is $\mathrm{FPdim}(\mathcal{R})=9u_2^2$.

\end{example}

\begin{proposition}\label{brayded}
There does not exist a pseudounitary braided fusion category $\mathcal{C}$ with $K(\mathcal{C})=K(\mathcal{R})$.
\end{proposition}

\begin{proof}
This proof is a finite computation; refer to Appendix \ref{a:1} for the SageMath \cite{sagemath} code used.  Assume there exists a pseudounitary braided fusion category $\mathcal{C}$ with $K(\mathcal{C})=K(\mathcal{R})$ whose simple objects will be indexed $X_0,\ldots,X_5$ as in Example \ref{one}.  We may assume $\mathcal{C}$ is equipped with the unique spherical structure such that $d_j:=\dim(X_j)=\mathrm{FPdim}(X_j)$ for $0\leq j\leq 5$ \cite[Proposition  8.23]{ENO}.  The symmetric center of $\mathcal{C}$ is integral \cite[Corollary 9.9.11]{tcat} hence trivial because the tensor unit is the only simple object of integer dimension.  Therefore $\mathcal{C}$ is modular.  Let $\sigma\in\mathrm{Gal}(\overline{\mathbb{Q}}/\mathbb{Q})$ be such that $\sigma(\cos(\pi/9))=\cos(5\pi/9)$.  Then one can verify on the basis of dimension alone, that $\hat{\sigma}(X_0)=X_3$, $\hat{\sigma}(X_3)=X_4$, and $\hat{\sigma}(X_4)=X_0$ where $\hat{\sigma}:\mathcal{O}(\mathcal{C})\to\mathcal{O}(\mathcal{C})$ is the permutation induced by the Galois action on the modular data of $\mathcal{C}$ \cite[Section 2.2.4]{paul}.
\par For each $X_j\in\mathcal{O}(\mathcal{C})$, the full twists $\theta_j:=\theta_{X_j}$ are roots of unity of order $3^a$ for some $a\in\mathbb{Z}_{\geq0}$ by \cite[Theorem 3.9]{paul}.  Hence $\gamma$, any cube root of the multiplicative central charge $\xi(\mathcal{C})$, is a root of unity of order $3^b$ for some $b\in\mathbb{Z}_{\geq0}$ as well.   Thus each normalized twist $t_j:=\theta_j/\gamma$ lies in $\mathbb{Q}(\zeta_{3^c})$ for some $c\in\mathbb{Z}_{\geq0}$.  But each pair $0\leq j\leq5$ and $\tau\in\mathrm{Gal}(\overline{\mathbb{Q}}/\mathbb{Q})$ satisfies $\tau^2(t_j)=t_{\hat{\tau}(j)}$ \cite[Theorem II(iii)]{dong2015congruence}.  Therefore $t_j$ is a root of unity of order 1, 3, or 9 or else $X_j$ has nine or more Galois conjugates.  Indeed, the subgroup of squares in $\mathrm{Gal}(\mathbb{Q}(\zeta_{3^n})/\mathbb{Q})$ has order $3^{n-1}$.  Moreover, $\theta_j,t_j\in\mathbb{Q}(\zeta_9)$ for all $0\leq j\leq 5$, hence $\gamma=\theta_j/t_j\in\mathbb{Q}(\zeta_9)$ as well.  The number of potential six-tuples $t_0,\ldots,t_5$ is then small; this also dictates $\gamma$ and $\theta_j$ as $t_0^{-1}=\gamma$.  In particular, $t_0$ determines $t_3$ and $t_4$ by Galois conjugacy.  Independently, if $t_1$ is one of the 6 primitive ninth roots of unity, then $t_2$ and $t_5$ are determined by Galois conjugacy.  Or there are $3^3$ possibilities for which \emph{all} $t_1,t_2,t_5$ are third roots of unity.  For each of these $9(6+3^3)=297$ potential six-tuples of normalized twists, we must have $\tau^+(\mathcal{C})\tau^-(\mathcal{C})=9u_2^2$ \cite[Proposition 8.15.4]{tcat} where $\tau^\pm(\mathcal{C})$ are the Gauss sums of $\mathcal{C}$.  Exactly two of the potential six-tuples $t_0,\ldots,t_5$ satisfy this constraint, whose corresponding twists are $\theta_0=1$, $\theta_1=\theta_2=\theta_5=\zeta_9$, $\theta_3=\zeta_3$, and $\theta_4=\zeta_3^2$, or these corresponding twists under the automorphism $\zeta_9\mapsto\zeta_9^{-1}$.
\par We now construct the $6\times6$ (unnormalized) $S$-matrices for these solutions with the \emph{balancing equation} for premodular categories \cite[Proposition 8.13.8]{tcat} which states for each integer pair $0\leq i,j\leq5$, we have $S_{ij}=\theta_i^{-1}\theta_j^{-1}\sum_{k=0}^5(N_i)_{j,k}d_k\theta_k$.  If $\mathcal{C}$ exists, these $S$-matrices satisfy the \emph{Verlinde formula} \cite[Corollary 8.14.4]{tcat} which states that $(N_i)_{j,k}=(1/9)u_2^{-2}\sum_{\ell=0}^5S_{i\ell}S_{j\ell}S_{k\ell}S_{0\ell}^{-1}$ for all $0\leq i,j,k\leq5$ because $\mathcal{C}$ is modular.  But $0=(N_1)_{1,1}\neq(1/9)u_2^{-2}\sum_{\ell=0}^5S_{1\ell}S_{1\ell}S_{1\ell}S_{0\ell}^{-1}=1$ in both cases.
\end{proof}

There are not as many tools available to study a pseudounitary fusion category if it is not braided.  So we will prove our main result by passing to the Drinfeld center $\mathcal{Z}(\mathcal{C})$ \cite[Section 8.5]{tcat} which is a modular tensor category, and will again possess a Galois action on simple objects.  The number theory dictating this Galois action comes from the cyclotomic field $\mathbb{Q}(\zeta_9)^+$ where $\mathrm{FPdim}(\mathcal{R})=9u_2^2$ lies.

\begin{proposition}\label{last}
Let $\mathcal{C}$ be a pseudounitary fusion category.  If $\mathrm{FPdim}(\mathcal{C})=9u_2^2$, then for all noninvertible $X\in\mathcal{O}(\mathcal{Z}(\mathcal{C}))$, $\mathbb{Q}(\dim(X))=\mathbb{Q}(\mathrm{FPdim}(X))=\mathbb{Q}(\zeta_9)^+$.
\end{proposition}

\begin{proof}
The category $\mathcal{Z}(\mathcal{C})$ is pseudounitary \cite[Remark 2.35]{DGNO} since $\mathcal{C}$ is.  The universal grading group of $\mathcal{Z}(\mathcal{C})$ is isomorphic to $\mathcal{O}(\mathcal{Z}(\mathcal{C})_\mathrm{pt})$ \cite[Theorem 6.3]{nilgelaki}, which has odd rank by \cite[Proposition 8.15]{ENO}.  Therefore the dimensional grading of $\mathcal{Z}(\mathcal{C})$ is trivial by \cite[Proposition 1.9]{2019arXiv191212260G}, i.e. \!$\dim(X)=\pm\mathrm{FPdim}(X)\in\mathbb{Q}(\zeta_9)^+$  for all $X\in\mathcal{O}(\mathcal{Z}(\mathcal{C}))$.  But $\mathbb{Q}(\zeta_9)^+$ is cubic, hence $\dim(X)\in\mathbb{Z}$ or $\mathbb{Q}(\dim(X))=\mathbb{Q}(\zeta_9)^+$ for all $X\in\mathcal{O}(\mathcal{Z}(\mathcal{C}))$.  We claim $\mathcal{Z}(\mathcal{C})$ has no noninvertible simple objects of integer dimension.  Indeed by \cite[Proposition 8.15]{ENO}, the fusion subcategory $\mathcal{D}\subset\mathcal{Z}(\mathcal{C})$ consisting of all objects of integer (Frobenius-Perron) dimension must have $\dim(\mathcal{D})\in\{1,3,3^2,3^3,3^4\}$ and any $X\in\mathcal{O}(\mathcal{D})$ must satisfy $\dim(X)^2\in\{1,3^2,3^4\}$ by \cite[Proposition 8.14.6]{tcat}.   But if there exists a simple object $X$ with squared integer dimension $\dim(X)^2\geq3^2$, then $\dim(\mathcal{D})\geq3^3$ because the tensor unit exists.  Hence $\dim(\mathcal{Z}(\mathcal{C}))/\dim(\mathcal{D})$, the global dimension of the relative centralizer of $\mathcal{D}$ in $\mathcal{Z}(\mathcal{C})$ \cite[Theorem 3.2(ii)]{mug1}, is either $3u_2^4$ or $u_2^4$.  But neither is totally greater than or equal to 1, violating \cite[Remark 2.5]{ENO}, so all simple objects of integer dimension in $\mathcal{Z}(\mathcal{C})$ are invertible.
\end{proof}


\section{Algebraic $d$-numbers in $\mathbb{Q}(\zeta_9)^+$ }\label{nums}

Here we classify totally positive algebraic $d$-numbers in $\mathbb{Q}(\zeta_9)^+$ of norm $3^n$ for $n\in\mathbb{Z}_{\geq0}$.  An \emph{algebraic $d$-number} is a non-zero algebraic integer $\alpha$ such that $\alpha/\sigma(\alpha)$ is an algebraic unit for all $\sigma\in\mathrm{Gal}(\overline{\mathbb{Q}}/\mathbb{Q})$ \cite[Lemma 2.7]{codegrees} where $\overline{\mathbb{Q}}$ is the algebraic closure of $\mathbb{Q}$.  If $\alpha$ is a cyclotomic algebraic $d$-number, then by the proof of \cite[Lemma 2.7(iv)]{codegrees},
\begin{equation}\label{def}
\alpha^{[\mathbb{Q}(\alpha):\mathbb{Q}]}=N_{\mathbb{Q}(\alpha)/\mathbb{Q}}(\alpha)u
\end{equation}
for some $u\in\mathcal{O}_{\mathbb{Q}(\alpha)}^\times$ where $N_{\mathbb{Q}(\alpha)/\mathbb{Q}}(\alpha):=\prod_{\sigma\in\mathrm{Gal}(\mathbb{Q}(\alpha)/\mathbb{Q})}\sigma(\alpha)$.  The Galois group $\mathrm{Gal}(\mathbb{Q}(\zeta_9)^+/\mathbb{Q})$ is cyclic of order 3, generated by $\sigma$ which acts by $\sigma(\cos(\pi/9))=\cos(5\pi/9)$, hence
\begin{equation}
\sigma(u_1)=-u_1^{-1}u_2,\qquad\text{and}\qquad\sigma(u_2)=u_1^{-1}.\label{fr}
\end{equation}


\subsection{Totally positive of norm $3^n$ for $n\in\mathbb{Z}_{\geq0}$}\label{totpos1}

\begin{lemma}\label{totpos}
Any totally positive algebraic unit in $\mathbb{Q}(\zeta_9)^+$ is equal to $(u_1^au_2^b)^2$ for some $a,b\in\mathbb{Z}$.
\end{lemma}

\begin{proof}
Any algebraic unit $u\in\mathbb{Q}(\zeta_9)^+$ is of the form $u=(-1)^\delta u_1^xu_2^y$ for some $x,y\in\mathbb{Z}$ and $\delta=0,1$.  From (\ref{fr}), the generator of the Galois group $\sigma\in\mathrm{Gal}(\mathbb{Q}(\zeta_9)^+/\mathbb{Q})$ acts on $u$ by
\begin{align}
\sigma(u)&=\sigma((-1)^\delta u_1^x u_2^y)=(-1)^{\delta+x}u_1^{-x-y}u_2^x,\text{ and} \\
\sigma^2(u)&=\sigma((-1)^{\delta+x}u_1^{-x-y}u_2^x)=(-1)^{\delta-y}u_1^yu_2^{-x-y}.
\end{align}
Both $u_1,u_2$ are positive, so for a unit $u\in\mathbb{Q}(\zeta_9)^+$ to be totally positive, we must have $\delta=0$, $x\equiv0\pmod{2}$, and $-y\equiv 0\pmod{2}$, implying our claim.
\end{proof}

\begin{lemma}\label{norm3}
Denote $\beta:=\sqrt[3]{3u_1^2u_2^2}=2-\zeta_9^4-\zeta_9^5$.  If $\alpha\in\mathbb{Q}(\zeta_9)^+$ is an algebraic $d$-number with $N_{\mathbb{Q}(\zeta_9)^+/\mathbb{Q}}(\alpha)=3$, then there exists $u\in\mathcal{O}_{\mathbb{Q}(\zeta_9)^+}^\times$ such that $\alpha=u\beta$.
\end{lemma}

\begin{proof}
We may assume that $\mathbb{Q}(\alpha)=\mathbb{Q}(\zeta_9)^+$ because for any integer $n$, $N_{\mathbb{Q}(\zeta_9)^+/\mathbb{Q}}(n)=n^3\neq3$.  If $\alpha\in\mathbb{Q}(\zeta_9)^+$ is an algebraic $d$-number with $N_{\mathbb{Q}(\zeta_9)^+/\mathbb{Q}}(\alpha)=3$, then $\alpha^3=3v$ for some $v\in\mathcal{O}_{\mathbb{Q}(\zeta_9)^+}^\times$ by (\ref{def}).  Hence $\alpha=\sqrt[3]{3v}$ and we aim to determine for which $v$, $\alpha\in\mathbb{Q}(\zeta_9)^+$.  Evidently, it suffices to consider the cases $v=u_1^xu_2^y$ where $x,y\in\{-2,-1,0,1,2\}$, otherwise one may remove cubic factors of the fundamental units from the radical.  Of these 25 cases, four pairs $(x,y)$ imply $\alpha\in\mathbb{Q}(\zeta_9)^+$: $(2,2)$, $(2,-1)$, $(-1,2)$, and $(-1,-1)$.  All four of these $\sqrt[3]{3u_1^xu_2^y}$ are unit multiples of $\beta:=\sqrt[3]{3u_1^2u_2^2}$.
\end{proof}

\begin{lemma}\label{norm9}
Denote $\beta:=\sqrt[3]{3u_1^2u_2^2}=2-\zeta_9^4-\zeta_9^5$.  If $\alpha\in\mathbb{Q}(\zeta_9)^+$ is an algebraic $d$-number with $N_{\mathbb{Q}(\zeta_9)^+/\mathbb{Q}}(\alpha)=9$, then there exists $u\in\mathcal{O}_{\mathbb{Q}(\zeta_9)^+}^\times$ such that $\alpha=u\beta^2$.
\end{lemma}

\begin{proof}
This follows from an identical argument as the proof of Lemma \ref{norm3}.  There are four integer pairs $x,y$ such that $\sqrt[3]{9u_1^xu_2^y}\in\mathbb{Q}(\zeta_9)^+$, which are all unit multiples of $\beta^2$.
\end{proof}

\begin{proposition}\label{furst}
Denote $\beta:=\sqrt[3]{3u_1^2u_2^2}=2-\zeta_9^4-\zeta_9^5$.   Any totally positive algebraic $d$-number in $\mathbb{Q}(\zeta_9)^+$ of norm $3^n$ for $n\in\mathbb{Z}_{\geq0}$ has a factorization as $3^a(u_1^bu_2^c)^2\beta^d$ for some $a\in\mathbb{Z}_{\geq0}$, $b,c\in\mathbb{Z}$, and $d=0,1,2$.
\end{proposition}

\begin{proof}
We have $\alpha=\sqrt[3]{3^nu_1^xu_2^y}$ for some integer pair $x,y$, as above.  But factoring out powers of $3^3$, we are faced with the classification problem of Lemma \ref{totpos}, Lemma \ref{norm3}, or Lemma \ref{norm9} as $\beta$ is totally positive.
\end{proof}


\subsection{Maximal among Galois conjugates}\label{max}

Recall $\beta:=\sqrt[3]{3u_1^2u_2^2}=2-\zeta_9^4-\zeta_9^5$ from Proposition \ref{furst}.  The Galois orbit of $\beta$ is $\{\beta,u_1^{-2}\beta,u_2^{-2}\beta\}$ (listed as $\beta$, $\sigma(\beta)$, $\sigma^2(\beta)$), therefore the Galois orbit of a totally positive $d$-number in $\mathbb{Q}(\zeta_9)^+$ of norm $3^n$ is
\begin{align}
\left\{3^a(u_1^bu_2^c)^2\beta^d,3^a(u_1^{-(b+c+d)}u_2^b)^2\beta^d,3^a(u_1^cu_2^{-(b+c+d)})^2\beta^d\right\}\label{ga}
\end{align}
by Proposition \ref{furst} for some $a\in\mathbb{Z}_{\geq0}$, $b,c\in\mathbb{Z}$, and $d=0,1,2$.  We will classify when $3^a(u_1^au_2^b)^2\beta^d$ is maximal among its Galois conjugates.  For this purpose, define $z:=\log(u_1)+\log(u_2)$,
\begin{align}
x_1:=&z^{-1}(\log(u_2)-2\log(u_1)),& y_{1,d}:=&dz^{-1}\log(u_1), \\
x_2:=&z^{-1}(\log(u_1)-2\log(u_2)),\text{ and}& y_{2,d}:=&dz^{-1}\log(u_2).
\end{align}
We have $3^a(u_1^bu_2^c)^2\beta^d$ maximal in its Galois orbit if and only if
\begin{align}
&&3^a(u_1^bu_2^c)^2\beta^d&\geq3^a(u_1^{-(b+c+d)}u_2^b)^2\beta^d \\
\Leftrightarrow&&(c-b)\log(u_2)&\geq-(2b+c+d)\log(u_1) \\
\Leftrightarrow&&cx_1^{-1}+y_{1,d}x_1^{-1}&\geq b,
\end{align}
as $x_1>0$, and by a symmetric argument, $b\geq cx_2-y_{2,d}$.
Note that $x_2-x_1^{-1}\neq0$, therefore
\begin{equation}
c\geq(y_{2,d}+y_{1,d}x_1^{-1})/(x_2-x_1^{-1}).\label{inew}
\end{equation}
Moreover $c$ is only constrained by the three inequalities (one for each $d=0,1,2$) in (\ref{inew}).  The three values the right-hand side of (\ref{inew}) attains are $0,-1/3,-2/3$, hence $c\in\mathbb{Z}_{\geq0}$ in any case.

\begin{proposition}\label{ex}
The set of all totally positive algebraic $d$-numbers in $\mathbb{Q}(\zeta_9)^+$ of norm $3^n$ for some $n\in\mathbb{Z}_{\geq0}$ which are maximal amongst their Galois conjugates is
\begin{equation}
\left\{3^a(u_1^b u_2^c)^2\beta^d:a,c\in\mathbb{Z}_{\geq0},b\in\mathbb{Z},d=0,1,2,\textnormal{ and }cx_1^{-1}+y_{1,d}x_1^{-1}\geq b\geq cx_2-y_{2,d}\right\}.
\end{equation}\label{s}
\end{proposition}


\subsection{Perfect squares}\label{square}

We wish to determine when $\alpha$ taken from the set in Proposition \ref{ex} is a perfect square in $\mathbb{Q}(\zeta_9)^+$, which is to say $\sqrt{\alpha}\in\mathbb{Q}(\zeta_9)^+$. This occurs if and only if $\sqrt{3^a\beta^d}\in\mathbb{Q}(\zeta_9)^+$.  As this condition is satisfied up to square powers of $3$ and $\beta$, we need only check that $\sqrt{\beta}\not\in\mathbb{Q}(\zeta_9)^+$, $\sqrt{3}\not\in\mathbb{Q}(\zeta_9)^+$, and $u_1^{-1}u_2^{-1}\beta^2=\sqrt{3\beta}\in\mathbb{Q}(\zeta_9)^+$.    In other words, $\alpha$ is a perfect square if and only if $a\equiv d\pmod{2}$.

\section{The basic data of $\mathcal{Z}(\mathcal{C})$}\label{sec:pos}

Assume $\mathcal{C}$ is a pseudounitary fusion category with $K(\mathcal{C})=K(\mathcal{R})$ and the simple objects of $\mathcal{C}$ are indexed $X_0,X_1,X_2,X_3,X_4,X_5$ as in Example \ref{one}.  With the results of Sections \ref{sec:fus} and \ref{nums}, we can determine the rank of $\mathcal{Z}(\mathcal{C})$ and the Frobenius-Perron/categorical dimensions of $X\in\mathcal{O}(\mathcal{Z}(\mathcal{C}))$ by computing their image under the forgetful functor $F:\mathcal{Z}(\mathcal{C})\to\mathcal{C}$.  To prevent confusion with the Galois orbits of simple objects in $\mathcal{Z}(\mathcal{C})$, we will name the set $\{\dim(\hat{\sigma}(X)):\sigma\in\mathrm{Gal}(\overline{\mathbb{Q}}/\mathbb{Q})\}$ the \emph{dimensional Galois orbit} of $X\in\mathcal{O}(\mathcal{Z}(\mathcal{C}))$.  The proof of Proposition \ref{dimensions} is relegated to Sections \ref{sec:frob}--\ref{sec:tech}.

\begin{proposition}\label{dimensions}
Let $\mathcal{C}$ be a pseudounitary fusion category with $K(\mathcal{C})=K(\mathcal{R})$.  There is a one-to-one correspondence between $X\in\mathcal{O}(\mathcal{Z}(\mathcal{C}))$ and the columns of Figure \ref{fig:W}.  The column corresponding to $X$ describes $F(X)$, where $F:\mathcal{Z}(\mathcal{C})\to\mathcal{C}$ is the forgetful functor.  The columns of Figure \ref{fig:W} are indexed by the dimensional Galois orbits in Figure \ref{fig:X} for ease of comparison.
\begin{figure}
\centering
\begin{equation*}
\begin{array}{|c||ccc|ccc|ccc|ccc|ccc|ccc|}
\hline
 &&o_0&&&o_1&&&o_1&&&o_1&&&o_1&&&o_1& \\\hline\hline
X_0&1&1&1&0&0&0&0&0&0&0&0&0&0&0&0&0&0&0 \\
X_1&0&1&1&1&1&0&1&1&0&0&0&1&0&0&1&0&0&1 \\
X_2&0&1&1&0&0&1&0&0&1&1&1&0&1&1&0&0&0&1 \\
X_3&0&1&2&0&1&1&0&1&1&0&1&1&0&1&1&0&1&1 \\
X_4&0&1&2&0&1&1&0&1&1&0&1&1&0&1&1&0&1&1 \\
X_5&0&1&1&0&0&1&0&0&1&0&0&1&0&0&1&1&1&0\\\hline
\end{array}
\end{equation*}
\begin{equation*}
\begin{array}{|c||ccc|ccc|ccc|c|c|c|c|c|c|c|c|c|}
\hline
  &&o_1&&&o_2&&&o_2&&o_4&o_4&o_4&o_5&o_5&o_5&o_5&o_5&o_5 \\\hline\hline
X_0&0&0&0&0&0&0&0&0&0&1&1&1&0&0&0&0&0&0 \\
X_1&0&0&1&0&1&0&0&1&0&1&0&0&1&1&0&0&0&0 \\
X_2&0&0&1&0&1&0&0&1&0&0&1&0&0&0&1&1&0&0 \\
X_3&0&1&1&1&1&0&1&1&0&1&1&1&0&0&0&0&0&0 \\
X_4&0&1&1&0&2&1&0&2&1&0&0&0&1&1&1&1&1&1 \\
X_5&1&1&0&0&1&0&0&1&0&0&0&1&0&0&0&0&1&1 \\\hline
\end{array}
\end{equation*}
\caption{$\dim\mathrm{Hom}(-,F(X))$ (rows) for $X\in\mathcal{O}(\mathcal{Z}(\mathcal{C}))$ (columns), separated by dimensional Galois orbit}%
    \label{fig:W}%
\end{figure}
\end{proposition}


\subsection{Frobenius-Perron constraint}\label{sec:frob}

Recall that the Galois orbit of $\dim(\mathcal{Z}(\mathcal{C}))$ is $3^4u_2^4$, $3^4u_1^{-4}$, and $3^4u_1^4u_2^{-4}$.  By Propositions \ref{furst}, \ref{ex}, and \cite[Theorem 1.8]{codegrees}, for all $X\in\mathcal{O}(\mathcal{Z}(\mathcal{C}))$, there exist $a,b,c,d\in\mathbb{Z}$ with $\dim(X)^2=3^{a}(u_1^bu_2^c)^2\beta^d$.  But $\mathrm{dim}(X)^2\mid\dim(\mathcal{Z}(\mathcal{C}))$ \cite[Proposition 8.14.6]{tcat} so $0\leq a\leq4$.  Recall the Galois orbit of $\dim(X)^2$ from (\ref{ga}).
Now the Galois action on the modular data of $\mathcal{Z}(\mathcal{C})$ implies
\begin{equation}\label{gaal}
\dim(\hat{\sigma}(X))^2=\sigma(\dim(X)^2)\dim(\mathcal{Z}(\mathcal{C}))/\sigma(\dim(\mathcal{Z}(\mathcal{C}))).
\end{equation}
Hence $a,b,c,d$ are constrained by
\begin{equation}
\dim(\hat{\sigma}(X))^2=3^a(u_1^{-(b+c+d)+2}u_2^{b+2})^2\beta^d\qquad\text{and}\qquad\dim(\hat{\sigma^2}(X))^2=3^a(u_1^{c-2}u_2^{-(b+c+d)+4})^2\beta^d.
\end{equation}
The constraint being that $\dim(X)^2$, $\dim(\hat{\sigma}(X))^2$, and $\dim(\hat{\sigma^2}(X))^2$ must be maximal among their Galois conjugates as they are Frobenius-Perron dimensions by \cite[Lemma 3.3(i)]{MR3632091}.  Hence the exponent of $u_2$ must be nonnegative by Proposition \ref{ex} which implies $c\geq 0$, $b\geq-2$, and $4\geq b+c+d$.  Moreover $0\leq c\leq6$ and $-2\leq b\leq4$.  We must have $\dim(X)^2$ is a perfect square in $\mathbb{Q}(\zeta_9)^+$ by Proposition \ref{last}, hence $a\equiv d\pmod{2}$ by Section \ref{square}.  This leaves a finite number of potential squared dimensions of simple objects to analyze.  


\subsection{Upper and lower bound}

From this section onward, we will use the abbreviated notation $[X,Y]:=\dim\mathrm{Hom}(X,Y)$ where $X,Y$ are objects of a fusion category $\mathcal{C}$.  The category $\mathcal{C}$ is omitted in this notation but we are only considering two categories ($\mathcal{C}$ and $\mathcal{Z}(\mathcal{C})$) so the ambient category will be clear from context.
\par Some of the potential squared dimensions of simple objects from Section \ref{sec:frob} are too large, while some are too small.  Applying \cite[Theorem 3.1.1]{ostrikremarks} to $X\in\mathcal{O}(\mathcal{D})$ where $\mathcal{D}$ is a modular tensor category, $\dim(X)^4\leq(1/2)\dim(\mathcal{D})\left(\dim(\mathcal{D})+1\right)$.  Specifically for $\mathcal{D}=\mathcal{Z}(\mathcal{C})$,
$\dim(X)^4\leq(1/2)81u_2^4(81u_2^4+1)<3938^2$.  As a lower bound, recall the forgetful functor $F:\mathcal{Z}(\mathcal{C})\to\mathcal{C}$ preserves dimension.  Hence for any $X\in\mathcal{O}(\mathcal{Z}(\mathcal{C}))$, $\dim(X)$ must decompose numerically into a $\mathbb{Z}_{\geq0}$-linear sum of $1$, $u_2$, $u_1u_2$, and $u_1^{-1}u_2^2$.  But precisely six $X\in\mathcal{O}(\mathcal{Z}(\mathcal{C}))$ contain the tensor unit of $\mathcal{C}$ in their image under the forgetful functor.

\begin{lemma}\label{fcd}
Let $\mathcal{C}$ be a pseudounitary fusion category with $K(\mathcal{C})=K(\mathcal{R})$.  There exist $\mathbbm{1}=Z_0,\ldots,Z_5\in\mathcal{O}(\mathcal{Z}(\mathcal{C}))$ such that $\dim(Z_1)=\dim(Z_2)=\dim(Z_5)=u_2^2$, $\dim(Z_3)=u_2^2u_2^2$, and $\dim(Z_4)=u_1^{-2}u_2^4$.  Moreover these are the unique $X\in\mathcal{O}(\mathcal{Z}(\mathcal{C}))$ with $[X_0,F(X)]\neq0$.
\end{lemma}

\begin{proof}
The fusion ring $K(\mathcal{C})$ is categorified by the modular tensor category $\mathcal{R}$, hence we may compute the formal codegrees of $\mathcal{C}$ as $\dim(\mathcal{C})/\dim(X)^2$ for all $X\in\mathcal{O}(\mathcal{R})$ \cite[Example 2.9]{ost15}.  In particular, the formal codegrees of $\mathcal{C}$ are $9,9,9,9u_2^2,9u_1^{-2},9u_1^2u_2^{-2}$.  The existence, uniqueness, and dimensions of the simple objects $Z_0,\ldots,Z_5$ in $\mathcal{Z}(\mathcal{C})$ corresponding to these formal codegrees are then implied by \cite[Theorem 2.13]{ost15}.
\end{proof}

It immediately follows that $(\mathcal{Z}(\mathcal{C}))_\mathrm{pt}$ is trivial so $Z_0,Z_3,Z_4$ are the unique simple objects of $\mathcal{Z}(\mathcal{C})$ of these dimensions.  On the other hand, there may exist $X\in\mathcal{O}(\mathcal{Z}(\mathcal{C}))$ with $\dim(X)=u_2^2$ and $[X_0,F(X)]=0$.  The goal of Lemma \ref{fcd} was to show that for any $X\in\mathcal{O}(\mathcal{Z}(\mathcal{C}))$ not isomorphic to $Z_0,\ldots,Z_5$, there exist $x,y,z\in\mathbb{Z}_{\geq0}$ such that $\dim(X)=\dim(F(X))=xu_2+yu_1u_2+zu_1^{-1}u_2^2$ where $x=[X_1\oplus X_2\oplus X_5,F(X)]$, $y=[X_3,F(X)]$ and $z=[X_4,F(X)]$.  Knowing $N_{\mathbb{Q}(\zeta_9)^+/\mathbb{Q}}(\dim(X))=3^n$ for some $n\in\mathbb{Z}_{\geq0}$, one verifies by checking $x,y,z\in\{0,1,2\}$ that the smallest possible dimensions of nontrivial $X\in\mathcal{O}(\mathcal{Z}(\mathcal{C}))$ are then (in increasing order) $u_2$, $u_1u_2$, $u_1^{-1}u_2^2$, and $u_1^{-1}u_2\beta$.  We record these four dimensions as well as their dimensional Galois orbits using Equation (\ref{gaal}), and note that any other $X\in\mathcal{O}(\mathcal{Z}(\mathcal{C}))$ must satisfy $\dim(X)^2\geq (u_1^{-1}u_2\beta)^2>49$.  A brute-force computer search reveals a small list of other candidate dimensions.
\par In detail, applying the constraints of this section to the finite list from Section \ref{sec:frob} produces only 20 potential Frobenius-Perron dimensions of simple objects $X\in\mathcal{Z}(\mathcal{C})$, computed using the SageMath \cite{sagemath} code in Appendix \ref{a:2} and listed in Figure \ref{fig:X} along with partial decompositions of $F(X)$ (the multiplicities of $X_1,X_2,X_5$ are ambiguous) in terms of the simple objects of $\mathcal{C}$.  We assign a unique numerical label to each type of simple object for future reference and group them by their dimensional Galois orbit.

\begin{figure}
\centering
\begin{equation*}
\begin{array}{|c|c|c|c|c|c|c|}
\hline\mathrm{Orbit} & \mathrm{Label} & \dim(X) & [X_0,F(X)] & [X_1\oplus X_2\oplus X_5,F(X)] & [X_3,F(X)] & [X_4,F(X)] \\\hline
&0&1 & 1 & 0 & 0 & 0 \\ 
0&1&u_1^2u_2^2 & 1 & 3 & 1 & 1 \\
&2&u_1^{-2}u_2^4 & 1 & 3 & 2 & 2 \\\hline
&3&u_2 & 0 & 1 & 0 & 0 \\
1&4&u_1u_2^2 & 0 & 1 & 1 & 1 \\
&5&u_1^{-1}u_2^3 & 0 & 2 & 1 & 1 \\\hline
&6&u_1u_2 & 0 & 0 & 1 & 0 \\
2&7&u_2^3 &  0 & 3 & 1 & 2 \\
&8&u_1^{-1}u_2^2 & 0 & 0 & 0 & 1 \\\hline
&9&u_1^{-1}u_2\beta & 0 & 1 & 1 & 0 \\
3&10&u_2\beta & 0 & 2 & 0 & 1 \\
&11&u_1^{-1}u_2^2\beta & 0 & 2 & 1 & 2 \\\hline
4&12&u_2^2 & 1 & 1 & 1 & 0 \\\hline
5&13&u_2^2 & 0 & 1 & 0 & 1 \\\hline
&14&3u_2 & 0 & 3 & 0 & 0 \\
6&15&3u_1u_2^2 & 0 & 3 & 3 & 3 \\
&16&3u_1^{-1}u_2^3 & 0 & 6 & 3 & 3 \\\hline
&17&u_1^{-1}\beta^2 & 0 & 0 & 1 & 1 \\
7&18&u_1^{-1}u_2\beta^2 & 0 & 3 & 2 & 2 \\
&19&u_1^{-2}u_2\beta^2 & 0 & 3 & 1 & 1 \\\hline
8&20&3u_2^2 & 0 & 3 & 0 & 3 \\\hline
\end{array}
\end{equation*}
    \caption{Potential decompositions of $F(X)$ for $X\in\mathcal{O}(\mathcal{Z}(\mathcal{C}))$, separated by dimensional Galois orbit}%
    \label{fig:X}%
\end{figure}


\subsection{Induction-restriction}\label{sec:indres}

Following \cite[Section 5.8]{ENO}, let $I:\mathcal{C}\to\mathcal{Z}(\mathcal{C})$ be the induction functor.  It is known that $F(I(X))\cong\oplus_{Y\in\mathcal{O}(\mathcal{C})}Y\otimes X\otimes Y^\ast$ for all $X\in\mathcal{O}(\mathcal{C})$ \cite[Proposition 5.4]{ENO}.  With the simple objects of $\mathcal{C}$ labeled $X_0,\ldots,X_5$ as in Example \ref{one}, we compute
\begin{align}
F(I(X_0))&=6X_0\oplus3X_1\oplus3X_2\oplus6X_3\oplus3X_4\oplus3X_5 \label{1} \\
F(I(X_1))&=3X_0\oplus15X_1\oplus6X_2\oplus12X_3\oplus15X_4\oplus6X_5 \\
F(I(X_2))&=3X_0\oplus6X_1\oplus15X_2\oplus12X_3\oplus15X_4\oplus6X_5 \\
F(I(X_3))&=6X_0\oplus12X_1\oplus12X_2\oplus24X_3\oplus21X_4\oplus12X_5 \label{fi1} \\
F(I(X_4))&=3X_0\oplus15X_1\oplus15X_2\oplus21X_3\oplus33X_4\oplus15X_5 \label{fi2} \\
F(I(X_5))&=3X_0\oplus6X_1\oplus6X_2\oplus12X_3\oplus15X_4\oplus15X_5\label{5}
\end{align}

Note that for all $X\in\mathcal{O}(\mathcal{Z}(\mathcal{C}))$, $[X_3,F(X)]$ and $[X_4,F(X)]$ are determined by the uniqueness of the dimensions of $X_3$ and $X_4$ in $\mathcal{O}(\mathcal{C})$.  From this, we can determine a finite list of possible ranks of $\mathcal{Z}(\mathcal{C})$ and dimensions of simple objects of $\mathcal{Z}(\mathcal{C})$ by computing potential decompositions of $I(X_3)$ and $I(X_4)$ into simple objects.  For example, if there exist $n\in\mathbb{Z}_{\geq0}$ isomorphism classes of $X\in\mathcal{Z}(\mathcal{C})$ with $\dim(X)=u_1^{-1}u_2\beta^2$, then by Galois conjugacy there also exist $n$ isomorphism classes of simple objects of dimensions $u_1^{-1}\beta^2$ and $u_1^{-2}u_2\beta^2$ (dimensional Galois orbit type $o_7$).  These $n$ triplets of isomorphism classes of simple objects, should they exist, contribute $X_3$ as a summand of $F(I(X_3))$ with multiplicity $(1^2+2^2+1^2)n=5n$ and $X_4$ as a summand of $F(I(X_4))$ with multiplicity $(1^2+2^2+1^2)=5n$.   In sum, over all triplets of dimensions indexed in Figure \ref{fig:X}, we must have $[X_3,F(I(X_3))]=24$ as in Equation (\ref{fi1}) and $[X_4,F(I(X_4))]=33$ as in Equation (\ref{fi2}).  Using the SageMath \cite{sagemath} code in Appendix \ref{a:3}, we compute that there are 45 possible decompositions of $I(X_3)$ and $I(X_4)$ which agree with the required values of $[X_3,F(I(X_3))]$ and $[X_4,F(I(X_4))]$.  Every dimensional Galois orbit contains at least one simple summand of $I(X_3)$ or $I(X_4)$, so these decompositions describe the entire set $\mathcal{O}(\mathcal{Z}(\mathcal{C}))$.  Note that no decompositions include dimensional Galois orbits of type $o_6$ in Figure \ref{fig:X}, so we disregard the possibility of isomorphism classes of simple objects of this type henceforth.


\subsection{Technical lemmas and the final computation}\label{sec:tech}

Here we develop a rudimentary understanding of the fusion rules of $\mathcal{Z}(\mathcal{C})$ to assist in computing possible decompositions of $I(X_j)$ for $j=1,2,5$ into simple objects of $\mathcal{Z}(\mathcal{C})$.  Throughout this section and in the SageMath \cite{sagemath} code in Appendix \ref{a:3}, we use the notation $n_{i,j,k}\in\mathbb{Z}_{\geq0}$ to denote the number of isomorphism classes of $X\in\mathcal{O}(\mathcal{Z}(\mathcal{C}))$ indexed by object label $i$ in Figure \ref{fig:X}, such that $F(X)$ contains $X_k$ as a summand with multiplicity $j$.  We will frequently use the facts that \cite[Proposition 2.10.8]{tcat} for simple objects $X,Y,Z$ in a fusion category $\mathcal{C}$, $[X,Y\otimes Z]=[X\otimes Y^\ast,\otimes Z]$ where the order of the tensor products is irrelevant to us because $K(\mathcal{C})$ and $K(\mathcal{Z}(\mathcal{C}))$ are commutative, and \cite[Section 5.8]{ENO} for $X\in\mathcal{O}(\mathcal{C})$ and $Y\in\mathcal{O}(\mathcal{Z}(\mathcal{C}))$, $[I(X),Y]=[X,F(Y)]$.

\begin{note}\label{ass}
Each of the results of this section tacitly assumes that the simple objects from Lemma \ref{fcd} decompose under the forgetful functor as $F(Z_0)=X_0$, $F(Z_j)=X_0\oplus X_j\oplus X_3$ for $j=1,2,5$, $F(X_3)=\oplus_{j=0}^5X_j$, and $F(X_4)=X_3\oplus X_4\oplus F(X_3)$.  Hence $Z_0,\ldots,Z_5$ are self-dual.  In the computations of Appendix \ref{a:3}, this assumption will be verified prior to applying the following technical lemmas.
\end{note}

\begin{lemma}\label{iud}
Let $\mathcal{C}$ be a pseudounitary fusion category with $K(\mathcal{C})=K(\mathcal{R})$.  Assume there exist $X,X'\in\mathcal{O}(\mathcal{Z}(\mathcal{C}))$ with $F(X)=F(X')=X_j$ for some $j=1,2,5$.  If $X^\ast\not\cong X'$, then $X\otimes X'\cong Z_k$, where $k=2$ if $j=1$, $k=5$ if $j=2$, and $k=1$ if $j=5$.  Moreover, if $X^\ast\cong X'$, then $[X\otimes X',Z_k]=0$.
\end{lemma}

\begin{proof}
The forgetful functor is monoidal, so we may compute
\begin{equation}
F(X\otimes X')=F(X)\otimes F(X')=X_j\otimes X_j=X_0\oplus X_k\oplus X_3.
\end{equation}
But $X^\ast\not\cong X'$ by assumption, so $[\mathbbm{1},X\otimes X']=0$.  The only non-unit simple objects of $\mathcal{Z}(\mathcal{C})$ which contain $X_0$ in their image under the forgetful functor are $Z_1,\ldots,Z_5$ by Lemma \ref{fcd}, and using Figure \ref{fig:X} the only possibility is that $X\otimes X'\cong Z_k$.  If $X^\ast\cong X'$, then $[\mathbbm{1},X\otimes X']=1$ thus $[X\otimes X',Z_k]=0$ as $\dim(X\otimes X')-1<\dim(Z_k)$. 
\end{proof}

\begin{lemma}\label{prev}
Let $\mathcal{C}$ be a pseudounitary fusion category with $K(\mathcal{C})=K(\mathcal{R})$.  Then for $j=1,2,5$,
\begin{itemize}
\item[(a)] $n_{3,1,j}<3$,
\item[(b)] if $n_{3,1,j}=1$, then $n_{9,1,j}>0$ and $n_{13,1,j}>0$, and
\item[(c)] if $n_{3,1,j}=2$, then either (i) $n_{4,1,j}>0$ and $n_{13,1,j}>0$, or (ii) $n_{11,2,j}>0$.
\end{itemize}
\end{lemma}

\begin{proof}
Let $j=1,2,5$.  We make no claims if $n_{3,1,j}=0$ so assume $n_{3,1,j}\geq1$, that is there exists $X\in\mathcal{O}(\mathcal{Z}(\mathcal{C}))$ with $F(X)=X_j$.  Without loss of generality we may assume $j=1$.  Our goal will be to understand the decomposition of $X\otimes Z_2$ where $Z_2$, defined in Lemma \ref{fcd}, satisfies $F(Z_2)=X_0\oplus X_2\oplus X_3$ under the assumptions of Note \ref{ass}.  On the level of the forgetful functor,
\begin{equation}
F(X\otimes Z_2)=F(X)\otimes F(Z_2)=X_1\otimes(X_0\oplus X_2\oplus X_3)=3X_1\oplus X_3\oplus 2X_4.
\end{equation}
We claim that $X_1$ is a subobject of $F(Y)$ for each simple summand $Y$ of $X\otimes Z_2$.  Indeed, if $[X\otimes Z_2,Y]>0$ for some $Y\in\mathcal{O}(\mathcal{Z}(\mathcal{C}))$, then $[X\otimes Y^\ast,Z_2]>0$.  But $[F(Z_2),X_0]=1$ hence $[F(X\otimes Y^\ast),X_0]=[X_1\otimes F(Y^\ast),X_0]>0$ which implies $[X_1,F(Y)]>0$ as well.  If $n_{3,1,1}>2$, there exist at least two nonisomorphic $X'$ and $X''$ such that $X^\ast\not\cong X'$ and $X^\ast\not\cong X''$, hence by Lemma \ref{iud}, $[X\otimes Z_2,X'^\ast]=[X\otimes X',Z_2]=1$ and likewise $[X\otimes Z_2,X''^\ast]=1$.  Moreover, there would exist a single remaining simple summand $Y$ of $X\otimes Z_2$ with $F(Y)=X_1\oplus X_3\oplus2X_4$ which does not exist by Figure \ref{fig:X}, proving claim (a).

\par If $n_{3,1,1}=1$, then there does not exist any $Y\in\mathcal{O}(\mathcal{Z}(\mathcal{C}))$ with $F(Y)=X_1$ and $[X\otimes Y,Z_2]>0$ by Lemma \ref{iud}.  Hence $X\otimes Z_2$ has no simple summands $Y$ with $F(Y)=X_1$.  Moreover, there is a unique decomposition of $X\otimes Z_2$ into simples.  In particular, there exist $P,Q_1,Q_2\in\mathcal{O}(\mathcal{Z}(\mathcal{C}))$ with
\begin{equation}
X\otimes Z_2\cong P\oplus Q_1\oplus Q_2,
\end{equation}
where $F(P)=X_1\oplus X_3$ and $F(Q_1)=F(Q_2)=X_1\oplus X_4$, proving (b).

\par Lastly, if $n_{3,1,1}=2$, there exists a unique simple summand $Y$ of $X\otimes Z_2$ with $F(Y)=X_1$ and we have
\begin{equation}
F((X\otimes Z_2)/Y)=2X_1\oplus X_3\oplus 2X_4.
\end{equation} 
If $(X\otimes Z_2)/Y$ is simple, then we have proven statement (c) part (ii), otherwise statement (c) part (i) follows because there does not exist $Z\in\mathcal{O}(\mathcal{Z}(\mathcal{C}))$ with $F(Z)=X_1\oplus2X_4$ by Figure \ref{fig:X}.
\end{proof}

\begin{lemma}\label{next}
Let $\mathcal{C}$ be a pseudounitary fusion category with $K(\mathcal{C})=K(\mathcal{R})$.  Then for $j=1,2,5$, if  $n_{3,1,j}>0$, then inclusively,
\begin{itemize}
\item[(a)] $n_{13,1,j}>0$ and $n_{5,1,j}>0$, or
\item[(b)] $n_{4,1,j}>0$ and $n_{10,1,j}>0$.
\end{itemize}
\end{lemma}

\begin{proof}
As in Lemma \ref{prev} it suffices to consider $X\in\mathcal{O}(\mathcal{Z}(\mathcal{C}))$ with $F(X)=X_1$.  We have
\begin{equation}
F(X\otimes Z_5)=F(X)\otimes F(Z_5)=2X_1\oplus X_3\oplus 2X_4\oplus X_5.
\end{equation}
But $X\otimes Z_5$ cannot be simple by Figure \ref{fig:X}.  So the image of each simple summand under the forgetful functor contains $X_1$ with multiplicity one as in the proof of Lemma \ref{prev}, and some nontrivial subobject of $X_3\oplus2X_4\oplus X_5$.  By Figure \ref{fig:X} the only two options, parenthesized by simple objects of $\mathcal{Z}(\mathcal{C})$, are
\begin{align}
F(X\otimes Z_5)&=(X_1\oplus X_4)\oplus(X_1\oplus X_3\oplus X_4\oplus X_5),\text{ or} \label{a}\\
F(X\otimes Z_5)&=(X_1\oplus X_3\oplus X_4)\oplus (X_1\oplus X_4\oplus X_5).\label{b}
\end{align}
Equation (\ref{a}) implies conclusion (a) and Equation (\ref{b}) implies conclusion (b).
\end{proof}

\begin{lemma}\label{ast}
Let $\mathcal{C}$ be a pseudounitary fusion category with $K(\mathcal{C})=K(\mathcal{R})$.  If $n_{3,1,j}\geq2$, then $n_{9,1,j}+o_2\geq2$ where $o_2$ is number dimensional Galois orbits of type 2 in $\mathcal{Z}(\mathcal{C})$ (refer to Figure \ref{fig:X}).
\end{lemma}

\begin{proof}
By assumption there exist $X,X'\in\mathcal{O}(\mathcal{Z}(\mathcal{C}))$ with $F(X)=F(X')=X_1$, without loss of generality, such that $X^\ast\not\cong X'$.  Then by Lemma \ref{iud} we have
\begin{equation}
1=[X\otimes X',X\otimes X']=[X\otimes X^\ast,X'\otimes(X')^\ast].
\end{equation}
But $X\otimes X^\ast$ and $X'\otimes(X')^\ast$ are isomorphic to either $\mathbbm{1}\oplus P\oplus Q$ with $F(P)=X_2$ and $F(Q)=X_3$, or $\mathbbm{1}\oplus R$ with $F(R)=X_2\oplus X_3$ for some $P,Q,R\in\mathcal{O}(\mathcal{Z}(\mathcal{C}))$ by Figure \ref{fig:X}, proving our claim.
\end{proof}

\begin{lemma}\label{laaast}
Let $\mathcal{C}$ be a pseudounitary fusion category with $K(\mathcal{C})=K(\mathcal{R})$.  If $n_{3,1,i}>0$, $n_{3,1,j}>0$, and $n_{3,1,k}=0$ for any combination of $i,j,k\in\{1,2,5\}$, then $n_{9,1,k}>0$.
\end{lemma}

\begin{proof}
Assume without loss of generality that $n_{3,1,1}>0$, $n_{3,1,5}>0$, and $n_{3,1,2}=0$.  Then there exists $X\in\mathcal{O}(\mathcal{Z}(\mathcal{C}))$ with $F(X)=X_1$ and there does not exist $Y\in\mathcal{O}(\mathcal{Z}(\mathcal{C}))$ with $F(Y)=X_2$.  Hence we must have $X\otimes X^\ast\cong\mathbbm{1}\oplus Z$ for some $Z\in\mathcal{O}(\mathcal{Z}(\mathcal{C}))$ where $F(Z)=X_2\oplus X_3$ by Figure \ref{fig:X}, proving our claim.
\end{proof}

With the technical lemmas proven, we now proceed with the final computation whose SageMath \cite{sagemath} code can be found in Appendix \ref{a:3}.  First, for each of the 45 dimensional Galois orbit decompositions computed in Section \ref{sec:indres}, denoted $o$, we create a list, denoted $\mathcal{L}_o$, of potential tuples $n_{i,j,k}\in\mathbb{Z}_{\geq0}$ for $k=1,2,5$ and $0\leq i\leq13$ \& $17\leq i\leq20$ (recall there are no simple objects of orbit type $o_6$) such that Equations (\ref{1})--(\ref{5}) are satisfied.  The computation of $\mathcal{L}_o$ is independent of $k=1,2,5$ since the decompositions of $F(I(X_1))$, $F(I(X_2))$ and $F(I(X_5))$ are symmetric under any permutation of $\{1,2,5\}$.  Next, for $k=1,2,5$, i.e. \!for each triplet $x,y,z\in\mathcal{L}_o$ ordered lexicographically, we verify that the number of simple objects implied by $x,y,z$ agrees with the orbit decomposition $o$.  The number of feasible triplets $x,y,z\in\mathcal{L}_o$ is then displayed and the assumption of Note \ref{ass} is verified with a warning displayed if this assumption of the technical lemmas does not apply.  Though this does not occur, it is necessary to check.  Lastly, consistency with Lemmas \ref{iud}--\ref{laaast} is verified and any feasible triplets are displayed.

\par There are exactly two possible solutions: one with $\mathrm{rank}(\mathcal{Z}(\mathcal{C}))=24$ and one with $\mathrm{rank}(\mathcal{Z}(\mathcal{C}))=36$.  The solution of rank 24 can be eliminated by a simple fusion argument.  The spurious solution has $I(X_j)$ decomposing into 15 simple summands for exactly one of $j=1,2,5$ with a summand $X\in\mathcal{O}(\mathcal{Z}(\mathcal{C}))$ such that $F(X)=X_j\oplus X_3$.  Assume without loss of generality that $j=1$.  We compute
\begin{equation}
F(X\otimes Z_1)=(X_1\oplus X_3)\otimes(X_0\oplus X_1\oplus X_3)=2X_0\oplus4X_1\oplus2X_2\oplus5X_3\oplus2X_4\oplus X_5.
\end{equation}
As $[X_0,F(Y)]$, $[X_1\oplus X_2\oplus X_5,F(Y)]$, $[X_3,F(Y)]$, and $[X_4,F(Y)]$ are determined for each $Y\in\mathcal{O}(\mathcal{Z}(\mathcal{C}))$ by Figure \ref{fig:X} (there are 13 distinct $F(Y)$ in this solution), one may verify by hand that there is no decomposition of $X\otimes Z_1$ into simple objects for this hypothetical solution.  The only remaining solution, of rank 36, has the unique decompositions of $F(X)$ for $X\in\mathcal{O}(\mathcal{Z}(\mathcal{C}))$ found in Figure \ref{fig:W}.


\section{The structure of $\mathcal{Z}(\mathcal{C})$}\label{sec:struc}

\begin{proposition}\label{prop:sub}
Let $\mathcal{C}$ be a pseudounitary fusion category.  If $K(\mathcal{C})=K(\mathcal{R})$, then there exists a braided fusion subcategory $\mathcal{D}\subset\mathcal{Z}(\mathcal{C})$ such that $K(\mathcal{D})=K(\mathcal{R})$.
\end{proposition}

\begin{proof}
By Proposition \ref{dimensions}, there exists $A_1\in\mathcal{O}(\mathcal{Z}(\mathcal{C}))$ with $F(A_1)=X_1$.  In Section \ref{sec:ding} below, we show $A_1$ generates a fusion subcategory containing the tensor unit, and five distinguished simple objects which we denote $A_1,A_2,A_3,B,C$.  In Section \ref{sec:sub} we prove that the set $\{\mathbbm{1},A_1,A_2,A_3,B,C\}$ is closed under tensor products/quotients, forming a fusion subcategory $\mathcal{D}$.  The fusion rules of $K(\mathcal{D})$ coincide with those of $\mathcal{R}$ from Example \ref{one} under the assignment $\mathbbm{1}\mapsto X_0$, $A_1\mapsto X_1$, $A_2\mapsto X_2$, $A_3,\mapsto X_5$, $B\mapsto X_3$ and $C\mapsto X_4$.
\end{proof}


\subsection{Five distinguished simple objects}\label{sec:ding}

Let $A_1\in\mathcal{O}(\mathcal{Z}(\mathcal{C}))$ be a simple object with $F(A_1)=X_1$.  As $F(A_1\otimes A_1^\ast)=X_0\oplus X_2\oplus X_3$ and $[A_1\otimes A_1^\ast,\mathbbm{1}]=1$, then we must have $A_1\otimes A_1^\ast=\mathbbm{1}\oplus A_2\oplus B_1$ for some $A_2,B_1\in\mathcal{O}(\mathcal{Z}(\mathcal{C}))$ with $F(A_2)=X_2$ and $F(B_1)=X_3$ as there does not exist $X\in\mathcal{O}(\mathcal{Z}(\mathcal{C}))$ with $F(X)=X_2\oplus X_3$ by Proposition \ref{dimensions}.  Symmetrically, we have
\begin{align}
A_1\otimes A_1^\ast&\cong\mathbbm{1}\oplus A_2\oplus B_1 \label{tf}\\
A_2\otimes A_2^\ast&\cong\mathbbm{1}\oplus A_3\oplus B_2,\text{ and} \\
A_3\otimes A_3^\ast&\cong\mathbbm{1}\oplus A_4\oplus B_3\label{tf2}
\end{align}
 for some $A_3,A_4,B_2,B_3\in\mathcal{O}(\mathcal{Z}(\mathcal{C}))$ with $F(A_3)=X_5$, $F(A_4)=X_1$, and $F(B_2)=F(B_3)=X_3$.  From isomorphisms (\ref{tf})--(\ref{tf2}), $A_2,A_3,A_4,B_1,B_2,B_3$ are all self-dual.  But there are only two isomorphism classes of $X\in\mathcal{O}(\mathcal{Z}(\mathcal{C}))$ with $F(X)=X_1$, hence this implies $A_1\cong A_1^\ast$ as well.  Note that $[A_j\otimes A_{j+1},A_j]=1$ from equivalences (\ref{tf})--(\ref{tf2}).  So we must have
\begin{align}
A_1\otimes A_2&\cong A_1\oplus C_1 \label{tf3}\\
A_2\otimes A_3&\cong A_2\oplus C_2,\text{ and} \label{tf35}\\
A_3\otimes A_4&\cong A_3\oplus C_3 \label{tf4}.
\end{align}
for some $C_1,C_2,C_3\in\mathcal{O}(\mathcal{Z}(\mathcal{C}))$ with $F(C_j)=X_4$ for $j=1,2,3$.  We may then compute for $j=1,2$,
\begin{equation}\label{eight}
2=[A_j\otimes A_{j+1},A_j\otimes A_{j+1}]=[A_j\otimes A_j,A_{j+1}\otimes A_{j+1}]=[\mathbbm{1}\oplus A_{j+1}\oplus B_j,\mathbbm{1}\oplus A_{j+2}\oplus B_{j+1}].
\end{equation}
But $A_{j+1}\not\cong A_{j+2}$ thus $B_1\cong B_2\cong B_3$ and we denote this distinguished isomorphism class by $B$.  The isomorphisms in (\ref{tf})--(\ref{tf2}) imply $[A_j\otimes B,A_j]=1$ for $j=1,2,3$.  So from $F(A_j\otimes B)=X_k\oplus X_3\oplus X_4$ where $k=2,3,1$ for $j=1,2,3$, respectively, $A_j\otimes B$ has three simple summands as there does not exist $X\in\mathcal{O}(\mathcal{Z}(\mathcal{C}))$ with $F(X)=X_3\oplus X_4$.  We then compute
\begin{equation}
3=[A_j\otimes B,A_j\otimes B]=[A_j\otimes A_j,B\otimes B]=[\mathbbm{1}\oplus A_{j+1}\oplus B,B\otimes B].\label{fortyeighta}
\end{equation}
But $F(B\otimes B)=\oplus_{j=0}^5X_j$ and (\ref{fortyeighta}) is independent of $j=1,2,3$ so we must have $[A_j,B\otimes B]=1$ for $j=2,3,4$ and therefore $[B,B\otimes B]=1$ as well.
Finally, this implies $[A_1,B\otimes B]=[A_1\otimes B,B]=1$, thus $A_1=A_4$.  To simplify future notation we will denote $A_2=A_5$ and $A_3=A_6$ as well.
\par We may now compute from (\ref{tf3})--(\ref{tf4}), 
\begin{equation}
[A_1\oplus C_1,A_2\oplus C_2]=[A_1\otimes A_2,A_2\otimes A_3]=[A_1\otimes A_3,A_2\otimes A_2]=[A_3\oplus C_3,\mathbbm{1}\oplus A_3\oplus B]=1.
\end{equation}
Thus $C_1\cong C_2$ and likewise $C_2\cong C_3$.  Denote this distinguished simple object by $C$.  For $j=1,2,3$, we have
\begin{equation}
[A_j\oplus C,A_j\otimes B]=[A_j\otimes A_{j+1},A_j\otimes B]=[A_j\otimes A_j,A_{j+1}\otimes B]=[\mathbbm{1}\oplus A_{j+1}\oplus B,A_{j+1}\otimes B]=2
\end{equation}
because $A_{j+1}\otimes B$ contains three simple summands: $A_{j+1}$, $B$, and another whose image under the forgetful functor is $X_4$.  Therefore $A_j\otimes B\cong A_j\oplus B\oplus C$ for $j=1,2,3$.  Lastly, we compute
\begin{equation}\label{newref2}
2=[A_1\otimes B,A_2\otimes B]=[A_1\otimes A_2,B\otimes B]=[A_1\oplus C,B\otimes B].
\end{equation}
We already know $[A_1,B\otimes B]=1$, thus $[C,B\otimes B]=1$ as well.  Moreover, we have completely determined the fusion rules between $A_1,A_2,A_3$ and $B$.  For $j=1,2,3$ (recalling that $A_1=A_4$), these fusion rules are summarized by:
\begin{align}
A_j\otimes A_j&\cong\mathbbm{1}\oplus A_{j+1}\oplus B, & A_j\otimes A_{j+1}&\cong A_j\oplus C \label{thurr0},\\
A_j\otimes B&\cong A_j\oplus B\oplus C,\qquad\text{ and } & B\otimes B&\cong\mathbbm{1}\oplus A_1\oplus A_2\oplus A_3\oplus B\oplus C.\label{thurr}
\end{align}
In particular, $C\cong C^\ast$.


\subsection{A fusion subcategory}\label{sec:sub}

We will proceed to show that $\mathbbm{1},A_1,A_2,A_3,B,C$ are the simple objects of a fusion subcategory of $\mathcal{Z}(\mathcal{C})$.  We need only describe the fusion rules of $C$ with the remaining objects using the results of Section \ref{sec:ding}.   From (\ref{thurr0}), for $j=1,2,3$, $[A_j\otimes C,A_{j+1}]=[A_j\otimes C,A_{j+2}]=1$ while $[A_j\otimes C,A_j]=0$, and $[A_j\otimes C,B]=1$ from (\ref{thurr}) (recalling that $A_1=A_4$ and $A_2=A_5$).  This determines $A_j\otimes C$ up to a single simple summand whose image under the forgetful functor is $X_4$.  We compute
\begin{equation}
[A_j\otimes A_{j+1},A_j\otimes C]=[(\mathbbm{1}\oplus A_{j+1}\oplus B)\otimes A_{j+1},C]=[\mathbbm{1}\oplus 2A_{j+1}\oplus A_{j+2}\oplus 2B\oplus C,C]=1.
\end{equation}
Therefore
\begin{equation}\label{thurrA}
A_j\otimes C\cong A_{j+1}\oplus A_{j+2}\oplus B\oplus C.
\end{equation}
Similarly from (\ref{thurr}) we have $[B\otimes C,A_j]=1$ for $j=1,2,3$ and $[B\otimes C,B]=1$.  This determines $B\otimes C$ up to two unknown summands whose images under the forgetful functor are both $X_4$.  We then compute
\begin{equation}
[A_j\otimes B,B\otimes C]=[B\otimes B,A_j\otimes C]=[\mathbbm{1}\oplus A_1\oplus A_2\oplus A_3\oplus B\oplus C,A_{j+1}\oplus A_{j+2}\oplus B\oplus C]=4.
\end{equation}
But we know $A_j\otimes B\cong A_j\oplus B\oplus C$, and $[B\otimes C,A_j]=[B\otimes C,B]=1$, so we must conclude $[B\otimes C,C]=2$.  This determines
\begin{equation}\label{thurrB}
B\otimes C\cong A_1\oplus A_2\oplus A_3\oplus B\oplus2C.
\end{equation}
From the above, we have $[C\otimes C,\mathbbm{1}]=1$, $[C\otimes C,A_j]=1$ for $j=1,2,3$ and $[C\otimes C,B]=2$.  This determines $C\otimes C$ up to two unknown summands whose images under the forgetful functor are both $X_4$.  We compute
\begin{equation}
[A_1\otimes A_2,C\otimes C]=[A_1\otimes C,A_2\otimes C]=[A_2\oplus A_3\oplus B\oplus C,A_1\oplus A_3\oplus B\oplus C]=3.
\end{equation}
Along with the fact that $[C\otimes C,A_1]=1$, this implies $[C\otimes C,C]=2$, thus
\begin{equation}\label{thurrC}
C\otimes C\cong\mathbbm{1}\oplus A_1\oplus A_2\oplus A_3\oplus 2B\oplus 2C.
\end{equation}
One should verify that the fusion rules determined in (\ref{thurr0}), (\ref{thurr}), (\ref{thurrA}), (\ref{thurrB}), and (\ref{thurrC}) coincide with those of $\mathcal{R}$ in Example \ref{one} under the assignment $\mathbbm{1}\mapsto X_0$, $A_1\mapsto X_1$, $A_2\mapsto X_2$, $A_3,\mapsto X_5$, $B\mapsto X_3$ and $C\mapsto X_4$.

\appendix
\section{SageMath code}\label{sage}

The following SageMath \cite{sagemath} code is written so that Appendices \ref{a:1}, \ref{a:2}, and \ref{a:3} can be run separately, or in sequence.  We include this code in a supplementary text file for ease of use.

\subsection{Proof of Proposition \ref{brayded}}\label{a:1}
\small
\begin{verbatim}
#Appendix A.1
print('*** Appendix A.1 ***')
q,G=CyclotomicField(9).gen(),CyclotomicField(9).automorphisms()
u1,u2,s=q-q^2-q^5,1-q^4-q^5,G[5]
d=[1,u2,u2,u1*u2,u1^(-1)*u2^2,u2]
N=[[[1,0,0,0,0,0],[0,1,0,0,0,0],[0,0,1,0,0,0],[0,0,0,1,0,0],[0,0,0,0,1,0],[0,0,0,0,0,1]],
   [[0,1,0,0,0,0],[1,0,1,1,0,0],[0,1,0,0,1,0],[0,1,0,1,1,0],[0,0,1,1,1,1],[0,0,0,0,1,1]],
   [[0,0,1,0,0,0],[0,1,0,0,1,0],[1,0,0,1,0,1],[0,0,1,1,1,0],[0,1,0,1,1,1],[0,0,1,0,1,0]],
   [[0,0,0,1,0,0],[0,1,0,1,1,0],[0,0,1,1,1,0],[1,1,1,1,1,1],[0,1,1,1,2,1],[0,0,0,1,1,1]],
   [[0,0,0,0,1,0],[0,0,1,1,1,1],[0,1,0,1,1,1],[0,1,1,1,2,1],[1,1,1,2,2,1],[0,1,1,1,1,0]],
   [[0,0,0,0,0,1],[0,0,0,0,1,1],[0,0,1,0,1,0],[0,0,0,1,1,1],[0,1,1,1,1,0],[1,1,0,1,0,0]]]
t1=[(q^a,(s^2)(q^a),(s^4)(q^a)) for a in range(9)]
t2=[(q^a,(s^2)(q^a),(s^4)(q^a)) for a in [1,2,4,5,7,8]]
R=range(3)
t2=t2+[(q^(3*x),q^(3*y),q^(3*z)) for x in R for y in R for z in R]
t0=[]
for x in t1:
    for y in t2:
        th=(1,y[0]/x[0],y[1]/x[0],x[1]/x[0],x[2]/x[0],y[2]/x[0])
        g=sum(d[z]^2*th[z] for z in range(6))
        if 9*u2**2==g*g.conjugate():
            t0.append(th)
thetas=list(set(t0))
R=range(6)
for t in thetas:
    S=[[t[x]^(-1)*t[y]^(-1)*sum(N[x][y][z]*d[z]*t[z]for z in R)for y in R]for x in R]
    def M(a,b,c):
        return (1/(9*u2^2))*sum([S[a][d]*S[b][d]*S[c][d]/S[0][d] for d in R])
    print('The fusion coefficient (N_1)_(1,1) should be 0, but it is', M(1,1,1))
\end{verbatim}

\normalsize
\subsection{Construction of Figure \ref{fig:X}}\label{a:2}
\small
\begin{verbatim}
#Appendix A.2
print('*** Appendix A.2 ***')
q=CyclotomicField(9).gen()
u1,u2,beta=q-q^2-q^5,1-q^4-q^5,2-q^4-q^5
def d1(a,b,c,d):
    return real(3^a*(u1^b*u2^c)^2*beta^d)
def d2(a,b,c,d):
    return d1(a,2-(b+c+d),b+2,d)
def d3(a,b,c,d):
    return d1(a,c-2,4-(b+c+d),d)
Z = log(u1)+log(u2)
X1,X2=Z^(-1)*(log(u2)-2*log(u1)),Z^(-1)*(log(u1)-2*log(u2))
def ub(a,b,c,d):
    return real(X1^(-1)*(c+d*Z^(-1)*log(u1)))
def lb(a,b,c,d):
    return real(c*X2-d*Z^(-1)*log(u2))
for a in range(5):
    for b in range(-2,5):
        for c in range(7):
            for d in range(3):
                if ((a-d)%2==0 and b+c+d<=4 and
                    ub(a,b,c,d)>=b>=lb(a,b,c,d) and
                    ub(a,2-(b+c+d),b+2,d)>=2-(b+c+d)>=lb(a,2-(b+c+d),b+2,d) and
                    ub(a,c-2,4-(b+c+d),d)>=c-2>=lb(a,c-2,4-(b+c+d),d)):
                    orbit = [d1(a,b,c,d),d2(a,b,c,d),d3(a,b,c,d)]
                    if max(orbit)<3938 and min(orbit)>49:
                        print((a,b,c,d), 'is a possible dim(X)^2=3^a(u1^b*u2^c)^2*beta^d')
dims=[[0,1],[1,u1^2*u2^2],[2,u1^(-2)*u2^4],[3,u2],[4,u1*u2^2],
      [5,u1^(-1)*u2^3],[6,u1*u2],[7,u2^3],[8,u1^(-1)*u2^2],[9,u1^(-1)*u2*beta],
      [10,u2*beta],[11,u1^(-1)*u2^2*beta],[12,u2^2],[13,u2^2],[14,3*u2],
      [15,3*u1*u2^2],[16,3*u1^(-1)*u2^3],[17,u1^(-1)*beta^2],
      [18,u1^(-1)*u2*beta^2],[19,u1^(-2)*u2*beta^2],[20,3*u2^2]]
print('Figure 3:')
for d in dims:
    for x in range(0,10):
        for y in range(0,10):
            for z in range(0,10):
                if (d[0] in [0,1,2,12] and
                    1+x*u2+y*u1*u2+z*u1^(-1)*u2^2==d[1]):
                    print(d[0], (1,x,y,z))
                else:
                    if x*u2+y*u1*u2+z*u1^(-1)*u2^2==d[1]:
                        print(d[0], (0,x,y,z))
\end{verbatim}


\normalsize
\subsection{Induction-restriction functors}\label{a:3}
\small

\begin{verbatim}
#Appendix A.3
print('*** Appendix A.3 *** (warning: slow)')
q=CyclotomicField(9).gen()
u1,u2,beta=q-q^2-q^5,1-q^4-q^5,2-q^4-q^5
dims=[[0,1],[1,u1^2*u2^2],[2,u1^(-2)*u2^4],[3,u2],[4,u1*u2^2],
      [5,u1^(-1)*u2^3],[6,u1*u2],[7,u2^3],[8,u1^(-1)*u2^2],[9,u1^(-1)*u2*beta],
      [10,u2*beta],[11,u1^(-1)*u2^2*beta],[12,u2^2],[13,u2^2],[14,3*u2],
      [15,3*u1*u2^2],[16,3*u1^(-1)*u2^3],[17,u1^(-1)*beta^2],
      [18,u1^(-1)*u2*beta^2],[19,u1^(-2)*u2*beta^2],[20,3*u2^2]]
B=81*u2^4-(1^2+(u1^2*u2^2)^2+(u1^(-2)*u2^4)^2)-3*(u2^2)^2
ods=[(u2)^2+(u1*u2^2)^2+(u1^(-1)*u2^3)^2,
     (u1*u2)^2+(u2^3)^2+(u1^(-1)*u2^2)^2,
     (u1^(-1)*u2*beta)^2+(u2*beta)^2+(u1^(-1)*u2^2*beta)^2,
     (u2^2)^2,
     (3*u2)^2+(3*u1*u2^2)^2+(3*u1^(-1)*u2^3)^2,
     (u1^(-1)*beta^2)^2+(u1^(-1)*u2*beta^2)^2+(u1^(-2)*u2*beta^2)^2,
     (3*u2^2)^2]
bounds=[floor(real(B/x))+1 for x in ods]
L=[]
test=[(1,o1,o2,o3,3,o5,o6,o7,o8) for o1 in range(bounds[0])
     for o2 in range(bounds[1]) for o3 in range(bounds[2])
     for o5 in range(bounds[3]) for o6 in range(bounds[4])
     for o7 in range(bounds[5]) for o8 in range(bounds[6])]
ir33=Matrix((1+2^2,1+1,1+1,1+1,1,0,3^2+3^2,1+2^2+1,0))
ir44=Matrix((1+2^2,1+1,2^2+1,1+2^2,0,1,3^2+3^2,1+2^2+1,3^2))
for x in test:
    if (24==Matrix(x)*ir33.transpose() and
        33==Matrix(x)*ir44.transpose()):
        L.append(x)
print('there are', len(L), 'possible decompositions of I(X_3) & I(X_4), hence sets O(Z(C))')

L=[(1,0,0,2,3,6,0,2,0),(1,0,1,1,3,6,0,2,0),(1,0,2,0,3,6,0,2,0),
   (1,1,0,1,3,0,0,2,1),(1,1,0,1,3,9,0,2,0),(1,1,0,4,3,0,0,1,0),
   (1,1,1,0,3,0,0,2,1),(1,1,1,0,3,9,0,2,0),(1,1,1,3,3,0,0,1,0),
   (1,1,2,2,3,0,0,1,0),(1,1,3,1,3,0,0,1,0),(1,1,4,0,3,0,0,1,0),
   (1,2,0,0,3,3,0,2,1),(1,2,0,0,3,12,0,2,0),(1,2,0,3,3,3,0,1,0),
   (1,2,1,2,3,3,0,1,0),(1,2,2,1,3,3,0,1,0),(1,2,3,0,3,3,0,1,0),
   (1,3,0,2,3,6,0,1,0),(1,3,1,1,3,6,0,1,0),(1,3,2,0,3,6,0,1,0),
   (1,4,0,1,3,0,0,1,1),(1,4,0,1,3,9,0,1,0),(1,4,0,4,3,0,0,0,0),
   (1,4,1,0,3,0,0,1,1),(1,4,1,0,3,9,0,1,0),(1,4,1,3,3,0,0,0,0),
   (1,4,2,2,3,0,0,0,0),(1,4,3,1,3,0,0,0,0),(1,4,4,0,3,0,0,0,0),
   (1,5,0,0,3,3,0,1,1),(1,5,0,0,3,12,0,1,0),(1,5,0,3,3,3,0,0,0),
   (1,5,1,2,3,3,0,0,0),(1,5,2,1,3,3,0,0,0),(1,5,3,0,3,3,0,0,0),
   (1,6,0,2,3,6,0,0,0),(1,6,1,1,3,6,0,0,0),(1,6,2,0,3,6,0,0,0),
   (1,7,0,1,3,0,0,0,1),(1,7,0,1,3,9,0,0,0),(1,7,1,0,3,0,0,0,1),
   (1,7,1,0,3,9,0,0,0),(1,8,0,0,3,3,0,0,1),(1,8,0,0,3,12,0,0,0)]
K=Matrix(((1,4,9,1,4,9,1,1,1,4,1,4,9,1,1,4,1,4,1,1,1,4,9,1,4,9,1,4,9),
          (2,2,0,2,2,0,0,0,1,0,2,2,0,0,1,0,1,0,0,0,2,2,0,2,2,0,2,2,0),
          (1,2,3,2,4,6,0,1,1,2,1,2,3,1,0,0,1,2,1,0,2,4,6,1,2,3,0,0,0),
          (1,2,3,2,4,6,0,1,1,2,2,4,6,0,1,2,2,4,0,1,2,4,6,1,2,3,3,6,9)))
for o in L:
    solutions=[]
    test=[(n11j,n12j,n13j,
           n21j,n22j,n23j,
           n31j,n41j,n51j,n52j,
           n71j,n72j,n73j,
           n91j,n101j,n102j,n111j,n112j,
           n121j,
           n131j,
           n181j,n182j,n183j,n191j,n192j,n193j,
           n201j,n202j,n203j)
           for n11j in range(o[0]+1)
           for n12j in range(o[0]-n11j+1)
           for n13j in range(o[0]-n11j-n12j+1)
           for n21j in range(o[0]+1)
           for n22j in range(o[0]-n21j+1)
           for n23j in range(o[0]-n21j-n22j+1)
           for n121j in range(o[4]+1)
           if n11j+n12j*2+n13j*3+n21j+n22j*2+n23j*3+n121j == 3
           for n31j in range(o[1]+1)
           for n41j in range(o[1]+1)
           for n51j in range(o[1]+1)
           for n52j in range(o[1]-n51j+1)
           for n71j in range(o[2]+1)
           for n72j in range(o[2]-n71j+1)
           for n73j in range(o[2]-n71j-n72j+1)
           for n91j in range(o[3]+1)
           for n101j in range(o[3]+1)
           for n102j in range(o[3]-n101j+1)
           for n111j in range(o[3]+1)
           for n112j in range(o[3]-n111j+1)
           for n131j in range(o[5]+1)
           for n181j in range(o[7]+1)
           for n182j in range(o[7]-n181j+1)
           for n183j in range(o[7]-n181j-n182j+1)
           for n191j in range(o[7]+1)
           for n192j in range(o[7]-n191j+1)
           for n193j in range(o[7]-n191j-n192j+1)
           for n201j in range(o[8]+1)
           for n202j in range(o[8]-n201j+1)
           for n203j in range(o[8]-n201j-n202j+1)
           if K*Matrix((n11j,n12j,n13j,n21j,n22j,n23j,n31j,
                        n41j,n51j,n52j,
                        n71j,n72j,n73j,
                        n91j,n101j,n102j,n111j,n112j,
                        n121j,
                        n131j,
                        n181j,n182j,n183j,n191j,n192j,n193j,
                        n201j,n202j,n203j)
                      ).transpose()==Matrix((15,12,12,15)).transpose()]
    M=[]
    for x in test:
        for y in test:
            if y <= x:
                for z in test:
                    if z <= y:
                        def s(j):
                            return x[j]+y[j]+z[j]
                        if ((1/3)*s(0)+(1/2)*s(1)+s(2)==o[0] and
                            (1/3)*s(3)+(1/2)*s(4)+s(5)==o[0] and
                            s(6)==o[1] and
                            s(7) == o[1] and
                            (1/2)*s(8)+s(9)==o[1] and
                            (1/3)*s(10)+(1/2)*s(11)+s(12)==o[2] and
                            s(13)==o[3] and
                            (1/2)*s(14)+s(15)==o[3] and
                            (1/2)*s(16)+s(17)==o[3] and
                            s(18)==o[4] and
                            s(19)==o[5] and
                            (1/3)*s(20)+(1/2)*s(21)+s(22)==o[7] and
                            (1/3)*s(23)+(1/2)*s(24)+s(25)==o[7] and
                            (1/3)*s(26)+(1/2)*s(27)+s(28)==o[8]):
                            M.append((x,y,z))
    if len(M)>0:
        print('(o_j)_0^8=', o, 'has', len(M), 'decomposition(s) of I(X_j) for j=1,2,5')
        N=[]
        for x in M:
            c=0
            if x[0][0]*x[0][3]*x[1][0]*x[1][3]*x[2][0]*x[2][3]!=1:
                print('***** WARNING: Technical lemmas do not apply *****')
            else:
                for y in x:
                    #Lemma 8
                    if (y[6]<2 or
                        y[6]>=2 and (a[2]+y[13])>=2):
                        #Lemma 6
                        if (y[6]==0 or
                            y[6]==1 and y[13]>0 and y[19]>0 or
                            y[6]==2 and y[7]>0 and y[19]>0 or
                            y[6]==2 and y[17]>0):
                            #Lemma 7
                            if (y[6]==0 or
                                y[6]>0 and y[8]>0 and y[19]>0 or
                                y[6]>0 and y[7]>0 and y[14]>0):
                                c=c+1
                if c==3:
                    #Lemma 9
                    if (x[0][6]*x[1][6]!=1 or x[2][6]!=0 or
                        x[2][13]>0 and (x[0][6]*x[1][6]==1 and x[2][6]==0)):
                        N.append(x)
        if len(N)>0:
            print('-->', len(N), 'decomposition(s) agree with the technical lemmas <---')
            for z in N:
                print(z)
        else:
            print('XXXXX no decomposition(s) agree with the technical lemmas XXXXX')
    else:
        print('(o_j)_0^8=', o, 'has no decompositions of I(X_j) for j=1,2,5')
\end{verbatim}
\normalsize

\bibliographystyle{plain}
\bibliography{bib}

\begin{thebibliography}{10}

\bibitem{MR3632091}
Paul Bruillard, Siu-Hung Ng, Eric~C. Rowell, and Zhenghan Wang.
\newblock On classification of modular categories by rank.
\newblock {\em Int. Math. Res. Not. IMRN}, (24):7546--7588, 2016.

\bibitem{paul}
Paul Bruillard, Siu-Hung Ng, Eric~C. Rowell, and Zhenghan Wang.
\newblock Rank-finiteness for modular categories.
\newblock {\em J. Amer. Math. Soc.}, 29(3):857--881, 2016.

\bibitem{MR866105}
Thomas~W. Cusick and Lowell Schoenfeld.
\newblock A table of fundamental pairs of units in totally real cubic fields.
\newblock {\em Math. Comp.}, 48(177):147--158, 1987.

\bibitem{dong2015congruence}
Chongying Dong, Xingjun Lin, and Siu-Hung Ng.
\newblock Congruence property in conformal field theory.
\newblock {\em Algebra Number Theory}, 9(9):2121--2166, 2015.

\bibitem{DGNO}
Vladimir Drinfeld, Shlomo Gelaki, Dmitri Nikshych, and Victor Ostrik.
\newblock On braided fusion categories. {I}.
\newblock {\em Selecta Math. (N.S.)}, 16(1):1--119, 2010.

\bibitem{tcat}
Pavel Etingof, Shlomo Gelaki, Dmitri Nikshych, and Victor Ostrik.
\newblock {\em Tensor categories}, volume 205 of {\em Mathematical Surveys and
  Monographs}.
\newblock American Mathematical Society, Providence, RI, 2015.

\bibitem{ENO}
Pavel Etingof, Dmitri Nikshych, and Viktor Ostrik.
\newblock On fusion categories.
\newblock {\em Ann. of Math. (2)}, 162(2):581--642, 2005.

\bibitem{2019arXiv191212260G}
Terry {Gannon} and Andrew {Schopieray}.
\newblock {Algebraic number fields generated by Frobenius-Perron dimensions in
  fusion rings}.
\newblock {\em arXiv e-prints}, page arXiv:1912.12260, December 2019.

\bibitem{nilgelaki}
Shlomo Gelaki and Dmitri Nikshych.
\newblock Nilpotent fusion categories.
\newblock {\em Adv. Math.}, 217(3):1053--1071, 2008.

\bibitem{MR3367967}
Yasuyuki Kawahigashi.
\newblock Conformal field theory, tensor categories and operator algebras.
\newblock {\em J. Phys. A}, 48(30):303001, 57, 2015.

\bibitem{mug1}
Michael M\"{u}ger.
\newblock On the structure of modular categories.
\newblock {\em Proc. London Math. Soc. (3)}, 87(2):291--308, 2003.

\bibitem{codegrees}
Victor Ostrik.
\newblock On formal codegrees of fusion categories.
\newblock {\em Math. Res. Lett.}, 16(5):895--901, 2009.

\bibitem{ost15}
Victor Ostrik.
\newblock Pivotal fusion categories of rank 3.
\newblock {\em Mosc. Math. J.}, 15(2):373--396, 405, 2015.

\bibitem{ostrikremarks}
Victor Ostrik.
\newblock Remarks on global dimensions of fusion categories.
\newblock In {\em Tensor categories and {H}opf algebras}, volume 728 of {\em
  Contemp. Math.}, pages 169--180. Amer. Math. Soc., Providence, RI, 2019.

\bibitem{MR1981895}
Viktor Ostrik.
\newblock Fusion categories of rank 2.
\newblock {\em Math. Res. Lett.}, 10(2-3):177--183, 2003.

\bibitem{rowell}
Eric~C. Rowell.
\newblock From quantum groups to unitary modular tensor categories.
\newblock In {\em Representations of algebraic groups, quantum groups, and
  {L}ie algebras}, volume 413 of {\em Contemp. Math.}, pages 215--230. Amer.
  Math. Soc., Providence, RI, 2006.

\bibitem{MR2414692}
Eric~C. Rowell.
\newblock Unitarizability of premodular categories.
\newblock {\em J. Pure Appl. Algebra}, 212(8):1878--1887, 2008.

\bibitem{MR2544735}
Eric~C. Rowell, Richard Stong, and Zhenghan Wang.
\newblock On classification of modular tensor categories.
\newblock {\em Comm. Math. Phys.}, 292(2):343--389, 2009.

\bibitem{MR4079742}
Andrew Schopieray.
\newblock Lie theory for fusion categories: {A} research primer.
\newblock In {\em Topological phases of matter and quantum computation}, volume
  747 of {\em Contemp. Math.}, pages 1--26. Amer. Math. Soc., Providence, RI,
  2020.

\bibitem{sagemath}
{The Sage Developers}.
\newblock {\em {S}ageMath, the {S}age {M}athematics {S}oftware {S}ystem
  ({V}ersion 9.1)}, 2020.
\newblock {\tt https://www.sagemath.org}.

\end{thebibliography}

\end{document}